\documentclass[11pt]{amsart}

\usepackage{fullpage}
\usepackage{amssymb}
\usepackage{latexsym}
\usepackage{amsmath}
\usepackage{amscd}
\usepackage{mathrsfs}
\usepackage{tikz}
\usetikzlibrary{matrix,arrows,decorations.pathmorphing}
\usepackage{amsthm}

\makeatletter
\@namedef{subjclassname@2020}{%
  \textup{2020} Mathematics Subject Classification}
\makeatother

\newtheorem{theorem}{Theorem}[section]

\theoremstyle{plain}

\newtheorem{conjecture}[theorem]{Conjecture}
\newtheorem{corollary}[theorem]{Corollary}
\newtheorem{lemma}[theorem]{Lemma}

\newtheorem{proposition}[theorem]{Proposition}

\theoremstyle{definition}
\newtheorem{definition}[theorem]{Definition}

\newtheorem{notation}[theorem]{Notation}

\theoremstyle{remark}

\begin{document}
\title{On the conjectures of Vojta and Campana over function fields with explicit exceptional sets}
\author{Natalia Garcia-Fritz}
\address{Facultad de Matematicas, Pontificia Universidad Catolica de Chile\newline 
\indent Campus San Joaquin, Avenida Vicuna Mackenna 4860\newline
\indent Santiago, Chile}
\email{natalia.garcia@mat.uc.cl}
\thanks{Acknowledgements: The author was supported by the grants ANID Fondecyt regular 1211004, ANID Fondecyt Iniciacion en Investigacion 11190172, and ANID PAI 79170039}
\subjclass[2020]{11J97,14H05,14C21}
\keywords{Vojta's conjectures, Campana's conjectures, function fields.}

\begin{abstract} We prove new cases of Vojta's conjectures for surfaces in the context of function fields, with truncation equal to one and providing an effective explicit description of the exceptional set. We also prove a general and explicit result towards Campana's conjecture over complex function fields of curves. Our methods rely on a local study of $\omega$-integral varieties.
\end{abstract}

\maketitle
\vspace{-1mm}


\section{Introduction}

In his thesis, Vojta \cite{Voj87} explored the relation between Nevanlinna Theory and results in Diophantine Approximation.
Motivated by this analogy, in parallel he formulated conjectures for number fields, function fields of complex curves, and complex meromorphic functions, which imply several deep results and conjectures in these areas. In the function field case, they involve a comparison between the \emph{height function} and a \emph{counting function} of a nonconstant morphism from a curve to a variety.
There is another version of these conjectures where the counting function is \emph{truncated at level one}, also proposed by Vojta (see Section 23 in \cite{Voj91}). This sharper version of the conjectures is considerably more difficult, see \cite{Voj98,Yam04,Gas10}. 

In all of its versions, Vojta's conjecture allows the existence of an \emph{exceptional subset} $Z$ where the conjectured height inequality does not need to hold, although it gives no hint of what this proper Zariski-closed subset should be. Finding the exceptional set is often a key problem. In the function field setting, for instance, let us mention that the analogue of the Schmidt subspace theorem has been obtained with an effective exceptional set by work of Wang \cite{Wan04}, but in general, results of this form are not common, specially in the case of truncation one.

In this work we prove new cases of Vojta's conjecture with truncated counting functions in the context of complex function fields (Theorem \ref{vojta}) with an effective explicit description of the exceptional set. We also prove a result towards Campana's conjecture over complex function fields of curves (Theorem \ref{Camp}) with an explicit exceptional set. We obtain Theorems \ref{vojta} and \ref{Camp} from the following result about $\omega$-integral curves for $\omega$ a section of the sheaf $(HS^m_{X/\mathbb{C}})_r$ of Hasse-Schmidt differential forms of order $m$ and degree $r$:

\begin{theorem}[Main Theorem]\label{main}
Let $C/\mathbb{C}$ be a smooth projective curve of geometric genus $g(C)$ and let $X/\mathbb{C}$ be a smooth projective variety of dimension at least two. Let $\phi:C\to X$ be a nonconstant morphism. Let $\mathcal{L}$ be an invertible sheaf on $X$, let $m,r>0$ be integers, let $\omega\in H^0(X,\mathcal{L}\otimes (HS^m_{X/\mathbb{C}})_r)$, and let $D_1,\ldots,D_q\subset X$ be smooth $\omega$-integral hypersurfaces such that $D=\sum_{i=1}^q D_i$ is a strict normal crossings divisor. If $\phi(C)$ is not contained in $\mathrm{supp}(D)$ and
\begin{eqnarray}\label{lades}
h_{X,\mathcal{O}(D)}(\phi)-h_{X,\mathcal{L}}(\phi)>\sum_{i=1}^q N^{(m)}(D_i,\phi)+2r\max\{0,g(C)-1\},
\end{eqnarray}
then $\phi(C)$ is an $\omega$-integral curve.
\end{theorem}

Note that $N^{(m)}(D,\phi)\leq m N^{(1)}(D,\phi)$, so one also has a version of this result for $N^{(1)}(D,\phi)$ with a different constant.
We remark that Theorem \ref{main} implies a version of Noguchi \cite{Nog97} and Wang \cite{Wan96} result, when $X=\mathbb{P}^n$ and $\omega$ is a Wronskian. We work out the case $n=2$ in Corollary \ref{NWcor} for exposition purposes. 

Inequality \eqref{lades} follows from a condition that makes a sheaf to be of negative degree on curves that intersect the target divisor $D$ with high multiplicity, which forces a certain group of global sections to be trivial, thus obtaining $\omega$-integrality for these curves. We can ``lower'' the degree of the relevant sheaf by a careful local computation; see Theorem \ref{dosvar}. Then we translate our bound to an inequality between heights and truncated counting functions (see Definitions \ref{hf} and \ref{tcn}). 

From Theorem \ref{main} we obtain a version of Vojta's conjecture with a description of the exceptional set by carefully studying intersections between pairs of $\omega$-integral curves in a surface, with $\omega$ a reduced symmetric differential form (see Definition \ref{reduced}). Intersections between these curves outside of the \emph{discriminant} $\Delta(\omega)$ of our symmetric differential (see Definition \ref{defdisc}) are well understood \cite{V,Gar18}, and here we complete this analysis by studying intersections between curves at points of $\Delta(\omega)$. We restrict to the case when the set $DP_X(\omega)$ of \emph{degenerate points} (see Definition \ref{defdeg}) does not intersect the target divisor, as we can always remove these finitely many points via blow-ups; see Lemma \ref{lemdp}.

Given an ample invertible sheaf $\mathcal{A}$ and an invertible sheaf $\mathcal{L}$ over a variety $X$, we define
$$\sigma(\mathcal{A},\mathcal{L})=\inf\left\{\frac{a}{b}:a,b\in\mathbb{Z}_{>0},\ \mathcal{A}^{\otimes a}\otimes(\mathcal{L}^\vee)^{\otimes b} \mbox{ is ample}\right\}.$$
The result we prove is the following:

\begin{theorem}\label{vojta}
Let $X/\mathbb{C}$ be a smooth projective surface. Let $r\geq 1$, let $\mathcal{L}$ be an invertible sheaf on $X$, and let $\omega\in H^0(X,\mathcal{L}\otimes S^r\Omega^1_{X/\mathbb{C}})$ be a nonzero reduced form.
Let $D_1,\ldots,D_q$ be distinct ample, smooth, irreducible $\omega$-integral curves not contained in $\Delta(\omega)$ such that $D=\sum_{i=1}^qD_i$ is a normal crossings divisor and $D\cap DP_X(\omega)=\emptyset$.

There is a proper Zariski-closed set $Z_{\omega,D}\subseteq X$ with the following property: Let $\epsilon>0$, and suppose that $$q>\frac{1}{\epsilon}\max\{2,\max_{1\leq i\leq q}\sigma(\mathcal{O}_X(D_i),\mathcal{L})\}.$$ For every smooth projective curve $C/\mathbb{C}$ and every nonconstant map $\phi:C\to X$ with $\phi(C)\not\subseteq Z_{\omega,D}$, we have
\begin{equation}\label{1k} (1-\epsilon)h_{X,\mathcal{O}(D)}(\phi)< \sum_{j=1}^q N^{(1)}(D_j,\phi)+2r\max\{0,g(C)-1\}.\end{equation}
More precisely, we can take $Z_{\omega,D}$ as the union of all $\omega$-integral curves passing through a point $P\in X$ which is in $D_i\cap D_j$ for some $i\neq j$, or in $D_i\cap \Delta(\omega)$ for some $i$.
\end{theorem}

The proof of Theorem \ref{vojta} gives an explicit and effective construction for the set $Z_{\omega,D}$. We illustrate it in the following particular case:

\begin{corollary}\label{quad2}
Let $\epsilon>0$ and let $L_1,\ldots,L_q\subseteq\mathbb{P}^2_\mathbb{C}$ be lines in a quadratic family with $q>4/\epsilon$. Write $D=\sum_{i=1}^q L_i$. 
Let $Y$ be the envelope of this quadratic family of lines, which is a conic in $\mathbb{P}^2$. For every smooth projective curve $C/\mathbb{C}$ and nonconstant morphism $\phi:C\to\mathbb{P}^2$ whose image is not contained in $Y\cup D$, we have
\begin{eqnarray*}
(1-\epsilon)h_{\mathbb{P}^2,\mathcal{O}(D)}(\phi)< \sum_{j=1}^q N^{(1)}(L_j,\phi)+4\max\{0,g(C)-1\}.\end{eqnarray*}
\end{corollary}

From Theorem \ref{main} we also obtain a result towards Campana's conjecture \cite{Cam05} in the context of function fields. Campana's conjecture has been formulated in great generality in \cite{AV18,Abr09} over number fields. Here the conjecture predicts algebraic degeneracy of Campana points over a number field under the assumption that $\mathcal{K}_X\otimes \mathcal{O}(D-D_{\overline{\epsilon}})$ is a big $\mathbb{Q}$-invertible sheaf. In the holomorphic case, Brotbek and Deng \cite{BD19} have recently proved Campana's conjecture when the target divisor is a general ample divisor of sufficiently large degree with respect to a fixed ample divisor. 
In the function field case \cite{RTW21} constructs examples of surfaces and divisors (analogous to \cite{CZ04}) where a version of Campana's conjecture holds. 
We now specify the notion of Campana curves that we will use:

Let $X/\mathbb{C}$ be a smooth variety of dimension $n\geq 1$. Let $D_1,...,D_q$ be different irreducible hypersurfaces of $X$ and let $D=\sum_{i=1}^q D_i$. Let $\epsilon_1,...,\epsilon_q>0$ be rational numbers, write $\overline{\epsilon} = (\epsilon_1,...,\epsilon_q)$ and define the $\mathbb{Q}$-divisor $D_{\overline{\epsilon}}=\sum_{i=1}^q \epsilon_jD_j$. Let $C\subseteq X$ be an irreducible curve and let $\nu_C:\tilde{C}\to X$ be its normalization. We say that $C$ is an $(X,D_{\overline{\epsilon}})$\emph{-Campana curve} if for every $1\leq j\leq q$ and every $P\in \tilde{C}$ we have that if $\mathrm{ord}_{P}(\nu_C^*E_{D_j,\nu_C(P)})\geq 1$ with $E_{D_j,\nu_C(P)}$ a local equation of $D_j$ at $\nu_C(P)$ then $\mathrm{ord}_{P}(\nu_C^*E_{D_j,\nu_C(P)})\geq 1/\epsilon_j$.
In other words, every time an $(X,D_{\overline{\epsilon}})$-Campana curve $C$ meets the components of $D$, it does it with large multiplicity controlled by the coefficients $\epsilon_j$. 

Let us state the version of Campana's conjecture over function fields that we consider. See Conjecture 8 in \cite{RTW21}, and see \cite{AV18} for the analogous statement over number fields.

\begin{conjecture}[Campana's conjecture over function fields] 
Let $X$ be a smooth projective variety over $\mathbb{C}$ and let $D$ be a strict normal crossings divisor on $X$ with irreducible components $D_1,...,D_q$. For each $1\le j\le q$ let $\epsilon_j>0$ be a rational number and let $\overline{\epsilon}=(\epsilon_1,...,\epsilon_q)$. Let $K_X$ be a canonical divisor on $X$. If the $\mathbb{Q}$-divisor $K_X+(D-D_{\overline{\epsilon}})$ is big, then there is a proper Zariski closed subset $Z_{\bar{\epsilon},\omega,D}\subseteq X$, an ample sheaf $\mathcal{A}$ on  $X$, and a constant $B$ depending on the previous data (in particular, on $\bar{epsilon}$ and $D$) such that for every $(X,D_{\overline{\epsilon}})$-Campana curve $C\subseteq X$ not contained in $Z_{\bar{\epsilon},\omega,D}$ we have
\begin{equation}\label{descamp}
h_\mathcal{A}(\nu_C)\leq B\max\{0, g(C)-1\}.
\end{equation}
\end{conjecture}

Note that the height inequality \eqref{descamp} is precisely the one required in Demailly's notion of algebraic hyperbolicity \cite{Dem97}.
In the direction of this conjecture, we prove:

\begin{theorem}\label{Camp} 
Let $X/\mathbb{C}$ be a smooth projective surface. Let $\mathcal{L}$ be an invertible sheaf on $X$ associated to a divisor $D_\mathcal{L}$, let $m,r$ be positive integers, let $\omega\in H^0(X,\mathcal{L}\otimes(HS^m_{X/\mathbb{C}})_r)$ be a non-zero reduced form, and let $D_1,...,D_q$ be irreducible smooth ample $\omega$-integral curves not contained in $\Delta(\omega)$ such that $D=\sum_{i=1}^q D_i$ is a strict normal crossings divisor with $D\cap DP_X(\omega)=\emptyset$. 

For each $1\leq j\leq q$ let $\frac{1}{m}>\epsilon_j>0$ be rational numbers and let $\overline{\epsilon}=(\epsilon_1,...,\epsilon_q)$. If $-D_\mathcal{L}+(D-D_{m\overline{\epsilon}})$ is big, then there is a proper Zariski closed subset $Z_{\omega,D}\subseteq X$, an ample sheaf $\mathcal{A}$ on  $X$, and a constant $B$ depending on the previous data such that for every $(X,D_{\overline{\epsilon}})$-Campana curve $C\subseteq X$ not contained in $Z_{\bar{\epsilon},\omega,D}$ we have
$$
h_\mathcal{A}(\nu_C)\leq B\max\{0, g(C)-1\}.
$$
In particular, all $(X,D_{\bar{\epsilon}})$-Campana curves of geometric genus $0$ or $1$ are contained in $Z_{\bar{\epsilon},\omega,D}$.

More concretely, the sheaf $\mathcal{A}$, the set $Z_{\bar{\epsilon},\omega,D}$, and the constant $B$ can be chosen in the following explicit way: Take any integer $M\geq 1$, ample divisor $A$, and effective divisor $E$ such that $$M(-D_\mathcal{L} +(D-D_{m\bar{\epsilon}}))\sim A+E,$$ which is possible by Kodaira's Lemma. With these choices, we can take $B=2Mr$, $\mathcal{A}=\mathcal{O}(A)$ and
 $$Z_{\bar{\epsilon},\omega,D}=\Delta(\omega)\cup\mathrm{supp}(E)\cup\mathrm{supp}(D)\cup V_{\omega,D},$$ with $V_{\omega,D}$ the union of  the finitely many $\omega$-integral curves passing through the points $P\in D\cap \Delta(\omega)$.
\end{theorem}

Let us note that in concrete examples one can explicitly compute $Z_{\bar{\epsilon},\omega,D}$, in a way similar to Corollary \ref{quad2}. 

To conclude this introduction, let us outline the structure of the manuscript. In Section \ref{2} we give preliminaries about Hasse-Schmidt differentials and $\omega$-integral curves, and in Section \ref{3} we prove some properties of $\omega$-integral curves that will be helpful for us. In Section \ref{4} we relate intersection multiplicities to $\omega$-integrality, presenting the local analysis that will be central in the proof of Theorem \ref{main}.
In Section \ref{5} we prove Theorem \ref{main} and in Section \ref{6} we prove Theorems \ref{vojta} and \ref{Camp}. In Section \ref{7} we prove Corollary \ref{quad2}.

\textbf{Acknowledgements:} I thank Hector Pasten, Jacob Tsimerman, and Julie Tzu-Yueh Wang for useful conversations that led to various improvements on earlier versions of this work. Part of this research was carried out while the author was a Postdoctoral Fellow at the University of Toronto. The author was supported by the grants ANID Fondecyt regular 1211004, ANID Fondecyt Iniciacion en Investigacion 11190172 and ANID PAI79170039.


\section{Definitions and preliminaries}\label{2}

We will work in the category of schemes over $\mathbb{C}$. Varieties are irreducible reduced separated schemes of finite type over $\mathbb{C}$. 

\subsection{Hasse-Schmidt differentials}

We briefly outline a theory of Hasse-Schmidt differentials, following the conventions in \cite{V2}. We follow this approach instead of the jet differentials one for its potential to work in varieties with singularities \cite{BTV19} and in positive characteristic \cite{CP21}.

\begin{definition}
Let $A$ be a ring, let $f\colon A\to B$ and $A\to R$ be $A$-algebras, and let $m\in \mathbb{N}$. A \emph{higher derivation of order $m$ from $B$ to $R$ over $A$} is a sequence $(D_0,\ldots,D_m)$, where $D_0\colon B\to R$ is an $A$-algebra homomorphism and $D_i\colon B\to R$ with $i=1,\ldots,m$ are homomorphisms of additive abelian groups such that
\begin{itemize}
\item $D_i(f(a))=0$ for all $a\in A$ and all $i=1,\ldots, m$,
\item (Leibniz rule) for $x,y\in B$ and $k=0,\ldots,m$, we have $D_k(xy)=\sum_{i+j=k}D_i(x)D_j(y).$
\end{itemize}
\end{definition}

The set of higher derivations of order $m$ from $B$ to $R$ will be denoted by $\mathrm{Der}_A^m(B,R)$. One has a covariant functor $\mathrm{Der}_A^m(B,\cdot)$ from the category of $A$-algebras to the category of sets.

\begin{definition}\label{hs}
Let $f\colon A\to B$ and $m$ be as above. Define the $B$-algebra of \emph{Hasse-Schmidt differentials} $HS^m_{B/A}$ to be the quotient of the polynomial algebra $B[x^{(i)}]_{x\in B,\ i=1,\ldots,m}$ by the ideal $I$ generated by the union of the sets
\begin{enumerate}
\item[(a)] $\{(x+y)^{(i)}-x^{(i)}-y^{(i)}:x,y\in B,\ i=1,\ldots,m\}$,
\item[(b)] $\{f(a)^{(i)}:a\in A,\ i=1,\ldots,m\}$,
\item[(c)] $\{(xy)^{(k)}-\sum_{i+j=k}x^{(i)}y^{(j)}:x,y\in B,\ k=0,\ldots,m\},$
\end{enumerate}
where we identify $x^{(0)}$ with $x$ for all $x\in B$. We also define the \emph{universal derivation $(d_0,\ldots,d_m)$ from $B$ to $HS^m_{B/A}$} by $d_ix=x^{(i)}\ (\mathrm{mod}\ I)$.
\end{definition}

The algebra $HS^m_{B/A}$ (either over $B$ or over $A$) is graded, the degree of $d_ix$ being $i$.
We denote by $(HS^m_{B/A})_r$ its homogeneous part of degree $r$.
The following is Proposition 1.6 in \cite{V2}.

\begin{proposition}\label{difv}
Let $A\to B$ and $A\to R$ be $A$-algebras, and let $m\in\mathbb{N}$. Given a derivation $(D_0,\ldots,D_m)$ from $B$ to $R$, there exists a unique $A$-algebra homomorphism $\phi\colon HS^m_{B/A}\to R$ such that $(D_0,\ldots,D_m)=(\phi\circ d_0,\ldots,\phi\circ d_m)$. Consequently $HS^m_{B/A}$, together with the universal derivation, represents the functor $\mathrm{Der}_A^m(B,\cdot)$.
\end{proposition}

Proposition \ref{difv} implies that for any homomorphism of $A$-algebras $f\colon C\to B$, the higher derivation $d_i\circ f\colon C\to HS^m_{B/A}$ over $A$ induces a map $HS^m_{C/A}\to HS^m_{B/A}$. Extending scalars to $B$ we obtain the homomorphism \begin{eqnarray*}f^{B/C}\colon HS^m_{C/A}\otimes_C B&\to& HS^m_{B/A}\cr d(c)\otimes1_B&\mapsto& d(f^\#(c)).\end{eqnarray*}
Note that $f^{B/C}$ is not related to the fundamental exact sequences presented in Section 2 of \cite{V2}.

The previous constructions behave well under localization and they give analogous notions for sheaves of $\mathcal{O}_X$-algebras on a scheme $X$, see Theorem 4.3 in \cite{V2}. 
Then, as in \cite{EGA} IV.16.5.3, by extending the definition of \emph{derivation from $\mathcal{O}_X$ to an $\mathcal{O}_X$-module} of \cite{EGA} IV.16.5.1 to the setting of higher derivations, we obtain a pair $(HS^m_{X/\mathbb{C}},(d_{X,0},\ldots,d_{X,m}))$ representing the functor $\mathrm{Der}^m_\mathbb{C}(\mathcal{O}_X,\cdot)$ of higher order derivations from $\mathcal{O}_X$ to an $
\mathcal{O}_X$-algebra over $\mathbb{C}$.
From a morphism of schemes $f:Y\to X$, we have a homomorphism $$f^{Y/X}\colon f^*HS^m_{X/\mathbb{C}}\to HS^m_{Y/\mathbb{C}}$$ locally defined by the maps $f^{B/C}$.

From now on, given a scheme $X$ we write $HS^m_X$ instead of $HS^m_{X/\mathbb{C}}$. We recall our convention that we are working in the category of schemes over $\mathbb{C}$.


\subsection{About $\omega$-integral subvarieties}

Given  an $\mathcal{O}_X$-module $\mathcal{F}$ and a morphism $f\colon Y\to X$, denote by $\rho^f_{\mathcal{F}}\colon \mathcal{F}\to f_*f^*\mathcal{F}$ the canonical homomorphism (\cite{EGA} 0$_{\mathrm{I}}$,3.5.3.2), and by $f^\#\colon \mathcal{O}_X\to f_*\mathcal{O}_Y$ the map defined in \cite{H}, p.\ 72.
For an invertible sheaf $\mathcal{L}$ on $X$, we define $$f^\bullet_{r,\mathcal{L}} \colon H^0(X,\mathcal{L}\otimes (HS^m_{X})_r)\to H^0(Y,f^*\mathcal{L}\otimes (HS^m_{Y})_r)$$ as the composition of the following maps
\begin{eqnarray*}
H^0(X,\mathcal{L}\otimes (HS^m_{X})_r) \to H^0(X,f_*f^*(\mathcal{L}\otimes (HS^m_{X})_r))&=& H^0(Y,f^*(\mathcal{L}\otimes (HS^m_{X})_r))\\
&\xrightarrow{\sim} & H^0(Y,f^*\mathcal{L}\otimes f^*(HS^m_{X})_r)\\
&\to & H^0(Y,f^*\mathcal{L}\otimes (HS^m_{Y})_r).
\end{eqnarray*}
Here the first map is induced by $\rho^f_{\mathcal{L}\otimes(HS^m_X)_i}$ and the last by $f^{Y/X}$.
An important property of this map that will be used in the proof of Theorem \ref{main} is the following: 

\begin{lemma}\label{comp}
Let $f:Y\to X$ and $g:Z\to Y$ be morphisms, and $\mathcal{L}$ an invertible sheaf on $X$. Then $(f\circ g)^\bullet=g^\bullet\circ f^\bullet$.
\end{lemma}

\begin{proof}
Working on affine open sets of $Z$, one verifies that the following diagram of $\mathcal{O}_Z$-modules is commutative
\begin{equation*}
\begin{gathered}
\begin{tikzpicture}[scale=1.3]
\node (A) at (0,0) {$g^*f^*(HS^m_X)_r$};
\node (B) at (3,0) {$g^*(HS^m_Y)_r.$};
\node (C) at (0,1) {$(f\circ g)^*(HS^m_X)_r$};
\node (D) at (3,1) {$(HS^m_Z)_r$};
\path[->,font=\scriptsize,>=angle 90]
(A) edge node[below] {} (B)
(C) edge node[left] {} (A)
(C) edge node[above] {} (D)
(B) edge node[right] {} (D);
\end{tikzpicture}
\end{gathered}
\end{equation*}
Then, applying Proposition 3.22 from \cite{thesis} we obtain a version of Proposition 3.33 in \cite{thesis}, namely that the following diagram is commutative
\begin{equation*}
\begin{gathered}
\begin{tikzpicture}[scale=1.3]
\node (A) at (0,0) {$g^*f^*(\mathcal{L}\otimes(HS^m_X)_r)$};
\node (B) at (8,0) {$g^*(f^*\mathcal{L}\otimes (HS^m_Y)_r).$};
\node (C) at (0,1) {$(f\circ g)^*(\mathcal{L}\otimes (HS^m_X)_r)$};
\node (D) at (8,1) {$g^*f^*\mathcal{L}\otimes (HS^m_Z)_r$};
\node (E) at (4,1) {$(f\circ g)^*\mathcal{L}\otimes(HS^m_Z)_r$};
\path[->,font=\scriptsize,>=angle 90]
(A) edge node[below] {} (B)
(E) edge node[below] {} (D)
(C) edge node[left] {} (A)
(C) edge node[above] {} (E)
(D) edge node[right] {} (B);
\end{tikzpicture}
\end{gathered}
\end{equation*}
By applying the functor $(f\circ g)_*=f_*\circ g_*$ and taking global sections we obtain the result.
\end{proof}

 We now introduce the fundamental definition:

\begin{definition}
Let $X$ be a smooth variety, let $m,r\geq 1$ be integers, let $\mathcal{L}$ be an invertible sheaf on $X$, and let $\omega\in H^0(X,\mathcal{L}\otimes(HS^m_{X})_r)$. A proper subvariety $Z\subseteq X$ is said to be $\omega$\emph{-integral} if the section $$(\nu_{Z})^\bullet_{r,\mathcal{L}}(\omega)\in H^0(\tilde{Z},\nu_Z^*\mathcal{L}\otimes (HS^m_{\tilde{Z}})_r)$$ is zero, where $\nu_Z\colon \tilde{Z}\to X$ is the normalization of $Z$.
\end{definition}


\section{Properties of $\omega$-integral varieties}\label{3}

\subsection{Global sections and $\omega$-integral curves}

From now on, by \emph{curve} we mean a projective algebraic variety of dimension one, not necessarily smooth.
Let $X$ be a smooth variety and consider an invertible sheaf $\mathcal{L}$ on $X$. Given a global section $\omega\in H^0(X,\mathcal{L}\otimes (HS^m_{X})_r)$, it follows that a curve $C\subset X$ is trivially $\omega$-integral when $H^0(\tilde{C},\nu_C^*\mathcal{L}\otimes (HS^m_{\tilde{C}})_r)=\{0\}.$
Here we give explicit numerical conditions for this group to be trivial.

\begin{lemma}\label{elgen1}
Let $Y$ be a smooth curve of genus $g\geq 1$ and let $r\in\mathbb{Z}_{>0}$. For any invertible sheaf $\mathcal{L}$ on ${Y}$ of degree $\deg_{Y}(\mathcal{L})<r(2-2g)$, and any $m\in\mathbb{N}$, we have $H^0(Y,\mathcal{L}\otimes (HS^m_{Y})_r)=\{0\}.$
\end{lemma}

\begin{proof} 
We know that for any $r\in\mathbb{N}$ $$\deg_{Y}(\mathcal{L}\otimes (HS^1_{Y})_r)=\deg_{Y}(\mathcal{L})+r(2g-2),$$ as $(HS^1_Y)_r=S^r\Omega^1_{Y}$. Since $\mathcal{L}$ satisfies $\deg_{Y}(\mathcal{L})<r(2-2g)$, we obtain $H^0(Y,\mathcal{L}\otimes (HS^1_{Y})_r)=0$.

Now consider $m>1$ such that the result is true for $m-1$, for all $r$ and all $\mathcal{L}$ satisfying $\deg_Y\mathcal{L}<r(2-2g)$. Fix $\mathcal{L}_0$ satisfying $\deg_{Y}(\mathcal{L}_0)<r(2-2g)$, and consider the exact sequences 
$$0\to \mathcal{L}_0\otimes S_{i-1}\to \mathcal{L}_0\otimes S_i\to\mathcal{L}_0\otimes S^i\Omega^1_{Y}\otimes (HS^{m-1}_{Y})_{r-mi}\to 0$$
for $1\leq i\leq \lfloor r/m\rfloor$ arising from the filtration (see \cite{GG}, paragraph (1.6)) $$(HS^{m-1}_{Y})_r=S_0\subseteq S_1\subseteq\cdots\subseteq S_{\lfloor r/m\rfloor}=(HS^{m}_{Y})_r,$$ 
where $\lfloor x\rfloor$ is the floor function. Taking global sections, for each $0\leq i\leq \lfloor r/m\rfloor$ we obtain
$$0\to H^0(Y,\mathcal{L}_0\otimes S_{i-1})\to H^0(Y,\mathcal{L}_0\otimes S_i)\to H^0(Y,\mathcal{L}_0\otimes S^i\Omega^1_{Y}\otimes (HS^{m-1}_{Y})_{r-mi}).$$
Note that $H^0(\mathcal{L}_0\otimes S_0)=\{0\}$ by induction hypothesis. Hence to make $H^0(Y,\mathcal{L}_0\otimes S_i)=\{0\}$ it suffices to have
\begin{eqnarray}\label{lo}
H^0(Y,\mathcal{L}_0\otimes S^i\Omega^1_Y\otimes (HS^{m-1}_Y)_{r-mi})=\{0\}
\end{eqnarray} 
for $0\leq i\leq\lfloor r/m\rfloor$. Since $g\geq 1$, for every $0\leq i\leq\lfloor r/m\rfloor$ we have $$\deg_{Y}(\mathcal{L}_0)<r(2-2g)\leq(r-(m-1)i)(2-2g).$$ 
It follows that
$$\mathrm{deg}_{Y}(\mathcal{L}_0\otimes S^i\Omega_Y^1)<(r-(m-1)i)(2-2g)+i(2g-2)=(r-mi)(2-2g).$$
By induction hypothesis with $\mathcal{L}=\mathcal{L}_0\otimes S^i\Omega^1_Y$ one obtains that Equation \eqref{lo} holds.
\end{proof}

\begin{lemma}\label{elgen2}
Let $Y$ be a smooth curve of genus zero and let $r\in\mathbb{Z}_{>0}$. For any invertible sheaf $\mathcal{L}$ on $Y$ with $\deg_{Y}(\mathcal{L})<0$, and any $m\in\mathbb{N}$, we have $H^0(Y,\mathcal{L}\otimes (HS^m_{Y})_r)=\{0\}.$
\end{lemma}

\begin{proof} 
The case $m=1$ is trivial. Now consider $m>1$ so that the result holds for $m-1$. As in Lemma \ref{elgen1}, to show that $H^0(Y,\mathcal{L}\otimes (HS^m_{Y})_r)=\{0\}$ for any $\mathcal{L}$ of negative degree, we need to prove that $H^0(Y,\mathcal{L}\otimes S^i\Omega^1_{Y}\otimes (HS^{m-1}_{Y})_{r-mi})=\{0\}$ for each $0\leq i\leq \lfloor r/m\rfloor$. As $Y$ has genus zero, by induction hypothesis this holds when 
$\deg_{Y}(\mathcal{L})-2i<0.$ 
Since $\deg_{Y}(\mathcal{L})<0$, we are done.
\end{proof}

From Lemma \ref{elgen1} and Lemma \ref{elgen2}, we obtain:

\begin{proposition}\label{corogen}
Let $Y$ be a smooth curve of genus $g$ and let $r\in\mathbb{Z}_{>0}$. For any invertible sheaf $\mathcal{L}$ on $Y$ with $\deg_{Y}(\mathcal{L})<2r\min\left\{0,1-g\right\}$, and any $m\in\mathbb{N}$, we have $H^0(Y,\mathcal{L}\otimes(HS^m_{Y})_r)=\{0\}.$
\end{proposition}


\subsection{Local study of $\omega$-integral subvarieties}

Let $f\colon Y\to X$ be a morphism of varieties, let $Q$ be a point in $Y$, and let $P=f(Q)\in X$. 
With $f^{\#}_Q\colon \mathcal{O}_{X,P}\to \mathcal{O}_{Y,Q}$ the map induced on stalks, let 
\begin{eqnarray*}
v_{f,Q}\colon HS^m_{\mathcal{O}_{X,P}/\mathbb{C}}&\to& HS^m_{\mathcal{O}_{Y,Q}/\mathbb{C}}\\
d_ia &\mapsto & d_if^{\#}_Q(a)
\end{eqnarray*}
 be the natural map of $\mathcal{O}_{X,P}$-algebras induced by $f^{\#}_Q$, and denote by $v_{f,Q,r}$ the restriction of $v_{f,Q}$ to $(HS^m_{\mathcal{O}_{X,P}/\mathbb{C}})_r$.
The purpose of this section is to relate the maps $f^\bullet$ and $v_{f,Q,r}$, so we can locally study $\omega$-integral subvarieties.

\begin{notation}
Given $f$ as above and a sheaf $\mathcal{G}$ on $X$, we denote by $f^\#_{\mathcal{G},Q}\colon\mathcal{G}_P\to (f^*\mathcal{G})_Q$ the induced homomorphism by $f$ on stalks.
\end{notation}

Let $Z$ be a variety, and $R$ a point in $Z$. For an invertible sheaf $\mathcal{F}$ on $Z$ and a trivialization $\gamma\colon\mathcal{F}_R\xrightarrow{\sim}\mathcal{O}_{Z,R}$, we define a map $\alpha_{R,\mathcal{F},\gamma}\colon H^0(Z,\mathcal{F}\otimes(HS_Z^m)_r)\to ((HS_Z^m)_r)_R$ as the composition of 
$$H^0(Z,\mathcal{F}\otimes(HS^m_Z)_r)\to (\mathcal{F}\otimes(HS^m_Z)_r)_R$$
with the functorial isomorphism $$(\mathcal{F}\otimes(HS^m_Z)_r)_R\xrightarrow{\sim} \mathcal{F}_R\otimes ((HS^m_Z)_r)_R,$$ and a trivialization isomorphism $$\mathcal{F}_R\otimes((HS^m_Z)_r)_R\xrightarrow{\sim} ((HS^m_Z)_r)_R$$ induced by $\gamma$. 

\begin{lemma}
If $Z$ is smooth at $R$, then $\alpha_{R,\mathcal{F},\gamma}$ is injective.
\end{lemma}

\begin{proof}
Since $Z$ is smooth, the sheaf $(HS^m_Z)_r$ is locally free. From this we immediately obtain that the homomorphism $H^0(Z,\mathcal{F}\otimes(HS^m_Z)_r)\to (\mathcal{F}\otimes(HS^m_Z)_r)_R$ is injective.
\end{proof}

\begin{lemma}\label{eps}
With $f$ as above, let $\mathcal{L}$ be an invertible sheaf on $X$. Consider a neighbourhood $U$ of $P$ such that there exists an isomorphism $\epsilon\colon\mathcal{L}|_{U}\to\mathcal{O}_{X}|_{U}$. The map $\epsilon$ induces isomorphisms $\epsilon_P\colon\mathcal{L}_P\to\mathcal{O}_{X,P}$ and $\epsilon_Q\colon(f^*\mathcal{L})_Q\to \mathcal{O}_{Y,Q}$ making the following diagram commutative
\begin{equation}\label{ma}
\begin{gathered}
\begin{tikzpicture}[scale=1.3]
\node (A) at (0,0) {$\mathcal{O}_{X,P}$};
\node (B) at (2,0) {$\mathcal{O}_{Y,Q}$};
\node (C) at (0,1) {$\mathcal{L}_P$};
\node (D) at (2,1) {$(f^*\mathcal{L})_Q$};
\path[->,font=\scriptsize,>=angle 90]
(A) edge node[below] {$f^\#_Q$} (B)
(C) edge node[left] {$\epsilon_P$} (A)
(C) edge node[above] {$f^\#_{\mathcal{L},Q}$} (D)
(D) edge node[right] {$\epsilon_Q$} (B);
\end{tikzpicture}
\end{gathered}
\end{equation}
\end{lemma}

\begin{proof} 
The isomorphism $\epsilon$ gives rise to an isomorphism at the level of stalks $\epsilon_P\colon \mathcal{L}_P\xrightarrow{\sim}\mathcal{O}_{X,P}$. On the other hand, $\epsilon$ gives rise to an isomorphism $f^*\mathcal{L}|_{f^{-1}(U)}\cong f^*\mathcal{O}_{X}|_{f^{-1}(U)}\cong \mathcal{O}_{Y}|_{f^{-1}(U)}$, obtaining an isomorphism $\epsilon_Q\colon (f^*\mathcal{L})_Q\xrightarrow{\sim}\mathcal{O}_{Y,Q}$ that makes Diagram \eqref{ma} commutative.
\end{proof}

\begin{lemma}\label{33} 
Let $\mathcal{L}$ be an invertible sheaf in $X$. Consider an isomorphism $\epsilon:\mathcal{L}|_{U}\xrightarrow{\sim}\mathcal{O}_{X}|_{U}$ as in Lemma \ref{eps}. The following diagram is commutative
\begin{equation*}
\begin{gathered}
\begin{tikzpicture}[scale=1.3]
\node (A) at (0,0) {$((HS^m_X)_r)_P$};
\node (B) at (6,0) {$((HS^m_Y)_r)_Q$};
\node (C) at (0,1) {$H^0(X,\mathcal{L}\otimes (HS^m_{X})_r)$};
\node (D) at (6,1) {$H^0(Y,f^*\mathcal{L}\otimes (HS^m_Y)_r)$};
\node (E) at (3,0) {$(f^*(HS^m_X)_r)_Q$};
\path[->,font=\scriptsize,>=angle 90]
(A) edge node[below] {$f^\#_{(HS^m_X)_r,P}$} (E)
(E) edge node[below] {$f^{Y/X}_Q$} (B)
(C) edge node[left] {$\alpha_{P,\mathcal{L},\epsilon}$} (A)
(C) edge node[above] {$f^\bullet_{r,\mathcal{L}}$} (D)
(D) edge node[right] {$\alpha_{Q,f^*\mathcal{L},\epsilon}$} (B);
\end{tikzpicture}
\end{gathered}
\end{equation*}
where $\alpha_{P,\mathcal{L},\epsilon}$, $\alpha_{Q,f^*\mathcal{L},\epsilon}$ are compatible by our choice of trivialization isomorphism.
\end{lemma}

\begin{proof} 
We study commutativity for each of the components of $\alpha_{P,\mathcal{L},\epsilon}$, $\alpha_{Q,f^*\mathcal{L},\epsilon}$.
The following diagram is commutative by the definition of $f^\bullet_{r,\mathcal{L}}$ (cf.\ Section 2.2) 
\begin{equation*}
\begin{gathered}
\begin{tikzpicture}[scale=1.3]
\node (A) at (0,0) {$(\mathcal{L}\otimes(HS^m_X)_r)_P$};
\node (B) at (0,1) {$H^0(X,\mathcal{L}\otimes(HS^m_X)_r)$};
\node (C) at (4,0) {$(f^*\mathcal{L}\otimes(HS_Y^m)_r)_Q.$};
\node (D) at (4,1) {$H^0(Y,f^*\mathcal{L}\otimes(HS^m_Y)_r)$};
\path[->,font=\scriptsize,>=angle 90]
(B) edge node[above] {$f^\bullet_{r,\mathcal{L}}$} (D)
(A) edge node[above] {} (C)
(B) edge node[above] {} (A)
(D) edge node[below] {} (C);
\end{tikzpicture}
\end{gathered}
\end{equation*}

The left hand side of the following diagram commutes by Proposition 3.24 in \cite{thesis}, the right hand side square commutes by distributivity of the stalk over tensor products.
\begin{equation*}
\begin{gathered}
\begin{tikzpicture}[scale=1.3]
\node (A) at (0,0) {$\mathcal{L}_P\otimes ((HS^m_X)_r)_P$};
\node (B) at (10,0) {$(f^*\mathcal{L})_Q\otimes ((HS^m_X)_r)_Q.$};
\node (C) at (0,1) {$(\mathcal{L}\otimes(HS^m_X)_r)_P$};
\node (D) at (10,1) {$(f^*\mathcal{L}\otimes (HS^m_Y)_r)_Q$};
\node (E) at (6.5,0) {$(f^*\mathcal{L})_Q\otimes(f^*(HS^m_X)_r)_Q$};
\node (F) at (3,1) {$(f^*(\mathcal{L}\otimes(HS^m_X)_r))_Q$};
\node (G) at (6.5,1) {$(f^*\mathcal{L}\otimes f^*(HS^m_X)_r)_Q$};
\path[->,font=\scriptsize,>=angle 90]
(A) edge node[below] {$f^\#_{\mathcal{L},Q}\otimes f^\#_{(HS^m_X)_r,Q}$} (E)
(E) edge node[below] {} (B)
(C) edge node[rotate=90,above] {$\sim$} (A)
(C) edge node[above] {} (F)
(F) edge node[above] {} (G)
(G) edge node[above] {} (D)
(G) edge node[rotate=90,above] {$\sim$} (E)
(D) edge node[rotate=90,below] {$\sim$} (B);
\end{tikzpicture}
\end{gathered}
\end{equation*}

Putting together the previous diagrams, and the following commutative diagram coming from the trivialization maps $\epsilon_P$, $\epsilon_Q$ from Lemma \ref{eps}
\begin{equation*}
\begin{gathered}
\begin{tikzpicture}[scale=1.3]
\node (A) at (0,0) {$((HS^m_X)_r)_P$};
\node (B) at (0,1) {$\mathcal{L}_P\otimes((HS^m_X)_r)_P$};
\node (D) at (4,1) {$(f^*\mathcal{L})_Q\otimes ((HS^m_Y)_r)_Q$};
\node (E) at (4,0) {$((HS^m_Y)_r)_Q$};
\path[->,font=\scriptsize,>=angle 90]
(A) edge node[below] {$f^{Y/X}_Q\circ f^\#_{(HS^m_Y)_r,Q}$} (E)
(B) edge node[rotate=90,above] {$\sim$} (A)
(D) edge node[rotate=90,below] {$\sim$} (E)
(B) edge node[right] {} (D);
\end{tikzpicture}
\end{gathered}
\end{equation*}
the result follows.
\end{proof}

Given a variety $Z$ and a point $R\in Z$, we consider the isomorphism defined by
\begin{eqnarray*}
\beta_{Z,R}\colon ((HS^m_Z)_r)_R &\xrightarrow{\sim}&    (HS^m_{\mathcal{O}_{Z,R}/\mathbb{C}})_r \\
\left(h \prod_{i=1}^m\prod_{j=1}^{k_i} (d_ig_{i,j})^{n_{i,j}}\right)_R &\mapsto& h_R\prod_{i=1}^m \prod_{j=1}^{k_i}(d_i(g_{i,j,R}))^{n_{i,j}},
\end{eqnarray*}
where $h,f,g_{i,j}$ are regular functions on a sufficiently small neighbourhood of $R$. Given an invertible sheaf $\mathcal{F}$ on $Z$, and $\epsilon:\mathcal{F}|_{U}\to \mathcal{O}_Z|_U$ we let $$\mu_{Z,R,\epsilon}=\beta_{Z,R}\circ\alpha_{R,\mathcal{F},\epsilon}.$$
 
\begin{lemma}\label{34}
With notation from Lemma \ref{33}, the following diagram is commutative
\begin{equation}\label{el}
\begin{gathered}
\begin{tikzpicture}[scale=1.3]
\node (A) at (0,1) {$((HS^m_X)_r)_P$};
\node (B) at (6,1) {$((HS^m_Y)_r)_Q$};
\node (C) at (0,0) {$(HS^m_{\mathcal{O}_{X,P}/\mathbb{C}})_r$};
\node (D) at (6,0) {$(HS^m_{\mathcal{O}_{Y,Q}/\mathbb{C}})_r$};
\node (E) at (3,1) {$(f^*(HS^m_X)_r)_Q$};
\path[->,font=\scriptsize,>=angle 90]
(A) edge node[below] {$f^\#_{(HS^m_X)_r,Q}$} (E)
(E) edge node[below] {$f_Q^{Y/X}$} (B)
(A) edge node[left] {$\beta_{X,P}$} (C)
(C) edge node[above] {$v_{f,Q,r}$} (D)
(B) edge node[right] {$\beta_{Y,Q}$} (D);
\end{tikzpicture}
\end{gathered}
\end{equation}
\end{lemma}

\begin{proof}  
By working on affine open sets, we check commutativity of Diagram \eqref{el} on elements.
\end{proof}

From Lemma \ref{33} and Lemma \ref{34} we can relate the maps $f^\bullet_{r,\mathcal{L}}$ and $v_{f,Q,r}$:

\begin{lemma}\label{32}
With notation from Lemma \ref{33}, the following diagram is commutative
$$\begin{CD}
H^0(X,\mathcal{L}\otimes (HS^m_X)_r) @>{f^\bullet_{r,\mathcal{L}}}>> H^0(Y,f^*\mathcal{L}\otimes (HS^m_Y)_r)\\
@V{\mu_{X,P,\epsilon}}VV @VV{\mu_{Y,Q,\epsilon}}V\\
(HS^m_{\mathcal{O}_{X,P}/\mathbb{C}})_r @>>{v_{f,Q,r}}> (HS^m_{\mathcal{O}_{Y,Q}/\mathbb{C}})_r.
\end{CD}$$
If $X$ and $Y$ are smooth, then both maps $\mu_{X,P,\epsilon}$ and $\mu_{Y,Q,\epsilon}$ are injective.
\end{lemma}


\section{Relation between intersection multiplicities and $\omega$-integrality}\label{4}

In this section we prove the key technical tool used in this manuscript, namely Theorem \ref{dosvar}. 

\begin{definition}
Let $X$ be a smooth variety. A divisor $D$ on $X$ has \emph{strict normal crossings} if it is reduced, if each irreducible component of its support is smooth and if those irreducible components intersect transversally, that is, if their defining equations are linearly independent in $m_{X,P}/m^2_{X,P}$ for each $P\in X$.
\end{definition}

\begin{definition}
Let $D$ be a divisor on a smooth variety $X$ and let $P\in X$. Let $E_{D,P}$ be a local equation of $D$ at $P$, where we consider $E_{D,P}=1$ if $P\not\in \mathrm{supp}(D)$. Given a nonconstant morphism from a smooth curve $\phi\colon C\to X$, for $Q\in C$ we define the \emph{vanishing order of $\phi$ and $D$ at $Q$}, as $\mathrm{ord}_Q(\phi^*E_{D,\phi(Q)})$, which is the exponent of a local parameter of $C$ at $Q$ in $\phi^*E_{D,\phi(Q)}$.\end{definition}

\begin{definition}
Given a smooth variety $X$, a locally free sheaf $\mathcal{F}$, an effective Cartier divisor $D$ on $X$ with associated subscheme $Y=Y_D$ and closed immersion $i_Y\colon Y\to X$, and a section $s\in H^0(X,\mathcal{F})$, we say that $s$\emph{ vanishes identically along }$D$ if the image of $s$ under the map $$H^0(X,\mathcal{F}) \to H^0(X,\mathcal{F}\otimes i_{Y*}\mathcal{O}_Y),$$ induced from $\mathcal{O}_X \to i_{Y*}\mathcal{O}_Y$, is zero.
\end{definition}

The following result will help us to relate $\omega$-integrality with the difference between the height and truncated counting functions. 

\begin{theorem}\label{dosvar} 
Let $C$ be a smooth curve and let $X$ be a smooth variety of dimension $n$. Let $\phi\colon C\to X$ be a nonconstant morphism, let $Q$ be a point in $C$, and $P=\phi(Q)$.
Let $m,r\in\mathbb{N}$, let $\mathcal{L}$ be an invertible sheaf on $X$, and let $\omega\in H^0(X,\mathcal{L}\otimes (HS^m_{X})_r)$. Let $D_1,\ldots,D_k$ be $\omega$-integral hypersurfaces forming a strict normal crossings divisor $D$. Suppose that $P\in\cap_{i=1}^k D_i$.

If $\mathrm{ord}_Q(\phi^*E_{D_i,P})=c_i>m$ for each $i=1,\ldots,k$, then the section $\phi^\bullet_{r,\mathcal{L}}\omega\in H^0(C,\phi^*\mathcal{L}\otimes (HS^m_{C})_r)$ vanishes identically along $(c_1+\cdots+c_k-km)Q$.
\end{theorem}

We will need some additional results. The following lemma is Corollary 3.84 in \cite{thesis}:

\begin{lemma}\label{ewi}
Let $D$ be a Cartier divisor on a smooth variety $X$ of the form $nD'$, with $D'$ prime and $n>0$. Let $Y$ be the associated subscheme of $D$, let $P\in Y$, let $\mathcal{F}$ be a locally free sheaf on $X$, and let $s\in H^0(X,\mathcal{F})$. If the image of $s_P\in\mathcal{F}_P$ by the map $\mathcal{F}_P\to (\mathcal{F}\otimes i_{Y*}\mathcal{O}_Y)_P$ is zero, then $s$ vanishes identically along $D$.
\end{lemma}

\begin{lemma}\label{dife}
Let $x,y\in B$, and let $n\geq i$ be nonnegative integers. We have that $d_i(x^ny)$ is a multiple of $x^{n-i}$.
\end{lemma}

\begin{proof} 
By the Leibniz rule in Definition \ref{hs} we have $d_2(x^n)=(xd_2x+(n-1)(dx)^2)x^{n-2}+xd_2x^{n-1}$. The result is easily proved by induction.
\end{proof}

\begin{lemma}\label{transv} 
Let $k\leq n$ be positive integers, let $X$ be a smooth variety of dimension $n$, and let $P\in X$ be a point.
If codimension one subvarieties $D_1,\ldots,D_k$ of $X$ intersect transversally at $P$ and $u_1,\ldots,u_k$ are their respective local equations in a neighbourhood of $P$, then $u_{1,P},\ldots,u_{k,P}$ form part of a system of local parameters of $X$ at $P$.
\end{lemma}

\begin{proof} 
Since the $D_j$ intersect transversally, the images of $u_{j,P}$ in $m_{X,P}/m^2_{X,P}$ are linearly independent. As the number of $u_{j,P}$ is less than or equal to $n$ and $X$ is smooth at $P$, the $\bar{u}_{j,P}\in m_{X,P}/m^2_{X,P}$ form part of a basis.
\end{proof}

\begin{lemma}\label{li}
Let $x_1,\ldots,x_n\in m_{X,P}$ be a system of local parameters of a smooth variety $X$ at a point $P$ and let $r>0$ be an integer. For each choice of $i_{jk}\in\mathbb{Z}_{>0}$ with $j\leq m$, $k\leq n$ satisfying $\sum_{\substack{j=1,\ldots,m\\ k=1,\ldots,n}} j\cdot i_{jk}=r$, consider $$\prod_{\substack{j=1,\ldots,m\\ k=1,\ldots,n}} d_{j}(x_{k})^{i_{jk}}\in (HS^m_{\mathcal{O}_{X,P}/\mathbb{C}})_r.$$
These elements form a basis of $(HS^m_{\mathcal{O}_{X,P}/\mathbb{C}})_r$ over $\mathcal{O}_{X,P}$. 
\end{lemma}

\begin{proof} 
Upon comparing completions at $P$, this reduces to the polynomial case covered in Proposition 5.1 of \cite{V2}. 
\end{proof}

\begin{proof}[Proof of Theorem \ref{dosvar}] 
Let $Z$ be the subscheme of $C$ associated to $(c_1+\cdots+c_k-km)Q$. Since $C$ is smooth, we have that $\phi^*\mathcal{L}\otimes (HS^m_{C})_r$ is locally free, hence by Lemma \ref{ewi} it is enough to prove that the image of $(\phi^\bullet_{r,\mathcal{L}}\omega)_{Q}$ under the map $$(\phi^*\mathcal{L}\otimes (HS^m_{C})_r)_{Q}\to (\phi^*\mathcal{L}\otimes (HS^m_{C})_r \otimes i_{Z*}\mathcal{O}_Z)_{Q}$$ is zero. 

Let $\tilde{P}_i\in \tilde{D}_i$ be the unique preimage of $P$ by the normalization map $\nu_{D_i}\colon \tilde{D}_i\to X$. Let $t_Q\in m_{C,Q}$ be a local parameter of $C$ at $Q$. Since $D$ is a strict normal crossings divisor, by Lemma \ref{transv} there is a system of local parameters $x_{1,P},x_{2,P},\cdots,x_{n,P}$ of $X$ at $P$ coming from rational functions $x_1,\ldots,x_n$ such that for each $i\leq k$, $x_i$ is a local equation for $D_i$ at $P$. Furthermore, $\ker(\nu^\#_{D_i,\tilde{P}_i})=(x_{i,P})$, and $\nu^\#_{D_i,\tilde{P}_i}(x_{j,P})$ with $j\neq i$ gives us a system of local parameters of $\tilde{D}_i$ at $\tilde{P}_i$. In particular $\nu^\#_{D_i,\tilde{P}_i}(x_{j,P})\neq 0$ when $j\neq i$.

Define the set of $m\times n$ matrices 
$$W=\left\{\textbf{h}: \textbf{h}_{pq}\in\mathbb{Z}_{\geq0}, \sum_{\substack{p=1,\ldots,m\\ q=1,\ldots,n}}\ell_1\textbf{h}_{pq}=r \right\}.$$
The entries of $\textbf{h}\in W$ are denoted by $\textbf{h}_{pq}$ for $1\leq p\leq m$, $1\leq q\leq n$. For $j=1,\ldots,n$ define the set $$A_j=\left\{\textbf{h}\in W: \textbf{h}_{1j}=\textbf{h}_{2j}=\cdots=\textbf{h}_{mj}=0\right\}.$$

Fix an isomorphism $\mathcal{L}_P\cong \mathcal{O}_{X,P}$. Since $P$ is a smooth point of $X$ and $x_{1,P},\ldots, x_{n,P}$ are local parameters at $P$, the image of $\omega$ by $\mu_{X,P}$ (defined before Lemma \ref{32}) in $(HS^m_{\mathcal{O}_{X,P}/\mathbb{C}})_r$ is of the form 
$$\omega_P=\sum_{\textbf{h}\in W} a_{\textbf{h}} \prod_{\substack{p=1,\ldots,m\\ q=1,\ldots,n}} d_{p}(x_{q,P})^{\textbf{h}_{pq}}$$

Now let $1\leq j\leq k$. Since $D_j$ is $\omega$-integral, by Lemma \ref{32} we know that $v_{\nu_{D_j},\tilde{P}_j,r}(\omega_P)=0$, thus 
\begin{eqnarray*}
0=v_{\nu_{D_j},\tilde{P}_j,r}(\omega_P)&=& \sum_{\textbf{h}\in {W}} \nu^\#_{D_j,\tilde{P}_j}(a_{\textbf{h}}) \prod_{\substack{p=1,\ldots,m\\ q=1,\ldots,n}} d_{p}\left(\nu^\#_{D_j,\tilde{P}_j}(x_{q,P})\right)^{\textbf{h}_{pq}}\\
&=& \sum_{\substack{\textbf{h}\in A_j}} \nu^\#_{D_j,\tilde{P}_j}(a_{\textbf{h}}) \prod_{\substack{p=1,\ldots,m\\ q=1,\ldots,n}} d_{p}\left(\nu^\#_{D_j,\tilde{P}_j}(x_{q,P})\right)^{\textbf{h}_{pq}},
\end{eqnarray*}
where the last equality follows as $\ker(\nu^\#_{D_j,\tilde{P}_j})=(x_{j,P})$. Since $\nu^\#_{D_j,\tilde{P}_j}(x_{i,P})$ with $j\neq i$ are a system of local parameters for $\tilde{D}_j$ at $\tilde{P}_j$, we have from this computation and Lemma \ref{li} applied on $\tilde{D}_j$ at $\tilde{P}_j$, that $\nu_{D_j,\tilde{P}_j}^\#(a_{\textbf{h}})=0$ for each $\textbf{h}\in A_j$. Therefore, for any $\textbf{h}\in A_j$ we obtain $a_{\textbf{h}}=g_{\textbf{h}}x_{j,P}$ for some $g_{\textbf{h}}\in\mathcal{O}_{X,P}$.

For $\textbf{h}\notin A_j$, at least one of the terms $\textbf{h}_{1j},\ldots ,\textbf{h}_{mj}$ is not zero, say $\textbf{h}_{sj}$. Therefore the term $a_{\textbf{h}} \prod_{\substack{p=1,\ldots,m\\ q=2,\ldots,n}} d_{p}(x_{q,P})^{\textbf{h}_{pq}}$ is a multiple of $d_s(x_{j,P})$. 

We conclude that for every $j\in\{1,\ldots,k\}$, each term  $a_{\textbf{h}} \prod_{\substack{p=1,\ldots,m\\ q=1,\ldots,n}} d_{p}(x_{q,P})^{\textbf{h}_{pq}}$ is a multiple of $d_{s_j}(x_{j,P})$ for some $s_j\in\{0,1,\ldots,m\}$, noting that we can take $s_j=0$ if $\textbf{h}\in A_j$. 

Since $\mathrm{ord}_Q(\phi^*E_{D_j,P})=c_j$, we have $\phi^\#_Q(x_{j,P})=u_jt_Q^{c_j}$ with $u_j\in\mathcal{O}_{C,Q}$ a unit. Therefore, in the ring $HS^m_{\mathcal{O}_{C,Q}/\mathbb{C}}$ each element $\phi_Q^\# (a_{\textbf{h}}) \prod_{\substack{p=1,\ldots,m\\ q=1,\ldots,n}} d_{p}(\phi_Q^\#(x_{q,P}))^{\textbf{h}_{pq}}$ is divisible by $\prod_{j=1}^k d_{s_j}(u_it^{c_j}_Q)$, which by Lemma \ref{dife} is divisible by $t_Q^{(c_1-m)+\cdots+(c_k-m)}$.
From Lemma \ref{32} we have
$$\mu_{C,Q}(\phi^\bullet_{r,\mathcal{L}}(\omega)) =v_{\phi,Q,r}(\omega_P)
= \sum_{\textbf{h}\in \overline{W}} \phi^\#_Q(a_{\textbf{h}}) \prod_{\substack{p=1,\ldots,m\\ q=1,\ldots,n}} d_{p}(\phi_P^\#(x_{q,P}))^{\textbf{h}_{pq}}
= t_Q^{c_1+\cdots+c_k-km}g,$$
for certain $g\in(HS^m_{\mathcal{O}_{C,Q}/\mathbb{C}})_r$.
Thus $(\phi^\bullet_{r,\mathcal{L}}\omega)_Q$ is a multiple of $t_Q^{c_1+\cdots +c_n-km}$. Therefore the image of $(\phi^\bullet_{r,\mathcal{L}}\omega)_{Q}$ in $(\phi^*\mathcal{L}\otimes (HS^m_C)_r\otimes i_{Z*}\mathcal{O}_Z)_Q$ is zero, as we wanted.
\end{proof}


\section{Proof of Theorem \ref{main} and a general height estimate}\label{5}

From now on, all varieties will be projective. In this section we prove Theorem \ref{main} and a version of Noguchi-Wang's theorem with variable truncation able to consider other families of hypersurfaces. We now briefly recall some definitions analogous to the ones in Nevanlinna theory:

\begin{definition}\label{hf}
Given a variety $X$, an invertible sheaf $\mathcal{L}$ on $X$, and a smooth curve $C$, the \emph{height function} is the map
\begin{eqnarray*}
h_{X,\mathcal{L}}\colon\mathrm{Mor}_{\mathrm{Var}}(C,X)&\to&\mathbb{Z}\\
\phi\quad&\mapsto&\deg_C(\phi^*\mathcal{L}).
\end{eqnarray*}
\end{definition}

We denote the height of $\phi$ by $h(\phi)$ if $X$ a projective space $\mathbb{P}^n$ and $\mathcal{L}\cong\mathcal{O}_X(1)$.

Given a smooth variety $X$, a nonconstant morphism $\phi\colon C\to X$ from a smooth curve $C$, and a divisor $D$ on $X$ such that $\phi(C)$ is not contained in $\mathrm{supp}(D)$, from the definition of the degree of a Cartier divisor on a curve it follows that $$h_{X,\mathcal{O}(D)}(\phi)=\sum_{Q\in C}\mathrm{ord}_Q(\phi^*E_{D,\phi(Q)}).$$

\begin{proposition}\label{eio}
Let $C$ be a smooth curve, let $X$ be a smooth variety, and $\phi\colon C\to X$ a nonconstant morphism. Let $\tilde{\phi}\colon C\to\widetilde{\phi(C)}$ be the lift of $\phi$ by the normalization map $\nu_{\phi(C)}\colon\widetilde{\phi(C)}\to X$. Let $Q\in C$, $P=\phi(Q)$, and $\tilde{P}=\tilde{\phi}(Q)$. If $D$ is a divisor on $X$ whose support does not contain $\phi(C)$,
then  
\begin{eqnarray}\label{ints}
\mathrm{ord}_Q(\phi^*E_{D,P})=e_{Q/\tilde{P}}(\tilde{\phi})\cdot \mathrm{ord}_{\tilde{P}}(\nu_{\phi(C)}^*E_{D,P}),
\end{eqnarray}
where $e_{Q/\tilde{P}}(\tilde{\phi})$ is the ramification index of $\tilde{\phi}$ at $Q$.
\end{proposition}

\begin{proof} 
Let $t$ be a local parameter of $\widetilde{\phi(C)}$ at $\tilde{P}$, and $s$ a local parameter of $C$ at $Q$. Recall that $e_{Q/\tilde{P}}(\tilde{\phi})$ is the exponent of $s$ in $\tilde{\phi}^*t$. Since $\phi^*E_{D,P}=\tilde{\phi}^*\nu_{\phi(C)}^*E_{D,P}$, we obtain Equation \eqref{ints}.
\end{proof}

With notation as in Proposition \ref{eio}, the usual intersection multiplicity of $\phi(C)$ and $D$ at a point $P\in \phi(C)$ can be written in terms of vanishing orders $$i_P(\phi(C),D)=\sum_{\tilde{P}|P}\mathrm{ord}_{\tilde{P}}(\nu_{\phi(C)}^*E_{D,P}).$$

\begin{definition}\label{tcn}
Given a smooth variety $X$, a smooth curve $C$, and a nonconstant morphism $\phi\colon C\to X$, for a prime divisor $D$ on $X$ such that $\phi(C)\not\subseteq \mathrm{supp}(D)$ the \emph{$n$-truncated counting function relative to $D$} is 
$$N^{(n)}(D,\phi)=\sum_{Q\in C}\min\{n,\mathrm{ord}_Q(\phi^*E_{D,\phi(Q)})\},$$ with $E_{D,\phi(Q)}$ a local equation of $D$ at $\phi(Q)$.
\end{definition}

\begin{proof}[Proof of Theorem \ref{main}:] 
Let $\phi(C)$ not contained in $\mathrm{supp}(D)$. By Theorem \ref{dosvar}, we know that $\phi^\bullet_{r,\mathcal{L}}\omega$ vanishes identically along $$S=\sum_{Q\in C}\sum_{i=1}^q\max\{\mathrm{ord}_Q(\phi^*E_{D_i,\phi(Q)})-m,0\}Q.$$

 For the sake of a contradiction, assume $\phi^\bullet_{r,\mathcal{L}}\omega\neq 0$.
With $Y_S$ the associated subscheme of $\mathcal{O}_C(-S)$ and $\mathcal{F}=\phi^*\mathcal{L}\otimes (HS^m_{C})_r$, from the exact sequence
$$0\to H^0(C,\mathcal{O}(-S)\otimes \mathcal{F})\to H^0(C,\mathcal{F})\to H^0(C,\mathcal{F}\otimes i_{S*}\mathcal{O}_{Y_S})$$
we obtain that there is a nonzero section $\omega_0\in H^0(C,\mathcal{O}(-S)\otimes \phi^*\mathcal{L}\otimes (HS^m_{C})_r)$ mapping to $\phi^\bullet_{r,\mathcal{L}}\omega$.
Adding $h_{X,\mathcal{O}(D)}(\phi)=\sum_{Q\in C}\sum_{i=1}^q\mathrm{ord}_Q(\phi^*E_{D_i,\phi(Q)})$ to the negative of Inequality \eqref{lades} gives us
\begin{eqnarray*}\label{smtaj}
\deg_C(\phi^*\mathcal{L})<\sum_{Q\in C}\sum_{i=1}^q\max\left\{\mathrm{ord}_{Q}(\phi^*E_{D_i,\phi(Q)})-m,0\right\}+2r\min\{0,1-g(C)\},
\end{eqnarray*}
thus by Proposition \ref{corogen} we obtain a contradiction. Therefore $\phi^\bullet_{r,\mathcal{L}}\omega$ must be zero. 

Write $C':=\widetilde{\phi(C)}$. Let $\nu_{\phi(C)}\colon C'\to X$ be the normalization of $\phi(C)$ and let $\tilde{\phi}\colon C\to C'$ be the lift of $\phi$. By Lemma \ref{comp}, since $\phi^\bullet_{r,\mathcal{L}}\omega=0$ we have that $\tilde{\phi}^\bullet_{r,\nu_{\phi(C)}^*\mathcal{L}} (\nu_{\phi(C)})^\bullet_{r,\mathcal{L}}\omega=0$. We now prove that $\tilde{\phi}^\bullet_{r,\nu^*_{\phi(C)}}$ is injective.
Since $C$ and $C'$ are smooth curves, the sheaves $(HS^m_{C'})_r$ and $(HS^m_C)_r$ are locally free. Hence $\mathrm{ker}(\rho^{\tilde{\phi}}_{(HS^m_{C'})_r})$ and $\mathrm{ker}(\tilde{\phi}^{C/C'})$ are locally free.

Since $\tilde{\phi}$ is a nonconstant morphism of curves, by checking at the generic point of $C$ we have that its rank is zero. Thus $$H^0(C',\nu_{\phi(C)}^*\mathcal{L}\otimes  (HS^m_{C'})_r)\to H^0(C,\tilde{\phi}^*(\nu_{\phi(C)}^*\mathcal{L}\otimes (HS^m_{C'})_r))\to H^0(C,\tilde{\phi}^*\nu_{\phi(C)}^*\mathcal{L}\otimes \tilde{\phi}^*(HS^m_{C'})_r)$$ is injective.

The sheaf $\mathrm{ker}(\tilde{\phi}^{C/C'})$ has rank zero by applying Proposition 5.9 of \cite{V2} to an open set where $\tilde{\phi}:C\to C'$ is \'etale. Since $\tilde{\phi}^*\nu_{\phi(C)}^*\mathcal{L}$ is invertible, we obtain that
$$H^0(C,\tilde{\phi}^*\nu_{\phi(C)}^*\mathcal{L}\otimes \tilde{\phi}^*(HS^m_{C'})_r)\to H^0(C,\tilde{\phi}^*\nu_{\phi(C)}^*\mathcal{L}\otimes (HS^m_{C})_r)$$
is also injective. Therefore $\tilde{\phi}^\bullet_{r,\nu_{\phi(C)}^*\mathcal{L}}$ is injective, thus $\phi(C)$ is an $\omega$-integral curve.
\end{proof}

The following result shows us that we can recover a version of Noguchi-Wang's theorem from Theorem \ref{main} in the case of surfaces. A similar approach using Wronskians proves the higher dimensional case. 

\begin{corollary}\label{NWcor} 
Let $C$ be a smooth curve. Let $\phi\colon C\to \mathbb{P}^2$ be a nonconstant morphism. Consider $L_1,\ldots,L_q$ to be lines in $\mathbb{P}^2$, no three meeting at a point. If $\phi(C)$ is not a line, then
\begin{equation}\label{ineq2}
(q-3)h(\phi)\leq \sum_{i=1}^q N^{(2)}(D_i,\phi)+6\max\{0,g(C)-1\}.
\end{equation}
\end{corollary}

\begin{proof} 
Consider $\omega\in H^0(\mathbb{P}^2,\mathcal{O}(3)\otimes (HS^2_{\mathbb{P}^2})_3)$ which locally looks like $dx_1d_2x_0-dx_0d_2x_1$ in the affine chart $\{x_2\neq 0\}$. Let $C'\subseteq\mathbb{P}^2$ be an $\omega$-integral curve intersecting the affine chart $\{x_2\neq 0\}$, and consider $V:=C'\cap\{x_2\neq 0\}\subseteq\mathbb{A}^2$. We will prove that $V$ is a line.

Assume that $V$ is not a vertical line. Let $P\in V$ be a smooth point such that the tangent of $V$ at $P$ is not vertical, that is, is not $x_0=0$. Let $\nu_V:\tilde{V}\to \mathbb{A}^2$ be its normalization. Since $V$ is smooth at $P$, the map $\nu$ is an isomorphism near $P$, and $\nu_{V,P}^\#:\mathcal{O}_{\mathbb{A}^2,P}\to\mathcal{O}_{\tilde{V},\tilde{P}}$ is also an isomorphism.

Let $t$ be a local parameter of $\tilde{V}$ at $\tilde{P}$. Since $V\subseteq\mathbb{A}^2$ is smooth at $P$ and its tangent is non-vertical, there exists $f(t)\in t\cdot k[[t]]$ such that $\nu_V$ induces 
\begin{eqnarray*}
\widehat{\nu_V}:k[[x_0,x_1]]\cong\hat{\mathcal{O}}_{\mathbb{A}^2,P}&\to& k[[t]]\cong \hat{\mathcal{O}}_{\tilde{V},\tilde{P}}\\
x_0&\mapsto&t\\
x_1&\mapsto&f(t).
\end{eqnarray*}
Write $f(t)=a_1t+a_2t^2+\cdots$. Consider the $t$-adic completion $(HS^2_{\tilde{V},\tilde{P}})_3\to (HS^2_{\tilde{V},\tilde{P}})_3\otimes\hat{\mathcal{O}}_{\tilde{V},\tilde{P}}\cong k[[t]](dt)^3\otimes k[[t]]dtd_2t=:M$. The image of $\nu_V^\bullet\omega$ in $M$ is $$\hat{\omega}:=\widehat{\nu_V^\bullet\omega_P}=dtd_2f-d_2tdf=\sum_{j=1}^\infty a_j(dtd_2(t^j)-d_2tdt^j).$$
By induction one proves that $dtd_2(t^j)-d_2tdt^j=\frac{j(j-1)}{2}t^{j-2}(dt)^3=\frac{1}{2}D_2(t^n)(dt)^3$, where $D_2$ is the usual second derivative. Therefore
$$\hat{\omega}=\sum_{j=1}^\infty a_j\frac{1}{2}D_2(t^j)(dt)^3=(a_2+a_33t+a_46t^2+\cdots)(dt)^3\in M.$$
Since $C'$ is $\omega$-integral, we have $\nu_V^\bullet\omega=0$, and we obtain $a_2=0, a_3=0,\ldots$. Therefore $f(t)=a_1t$ and $V$ is a line. By checking that all lines in $\mathbb{P}^2$ are $\omega$-integral, we conclude that the $\omega$-integral curves are the lines in $\mathbb{P}^2$.
Inequality \eqref{lades} in this setting becomes
$$(q-3)h(\phi)>\sum_{i=1}^qN^{(2)}(D_i,\phi)+6\max\{0,g(C)-1\}.$$
 
If Inequality \eqref{ineq2} holds, then by Theorem \ref{main} the curve $\phi(C)$ is $\omega$-integral and hence a line.
\end{proof}


\section{Intersection between $\omega$-integral curves and the proof of Theorems \ref{vojta} and \ref{Camp}}\label{6}

To give an explicit description of the exceptional set in Theorems \ref{vojta} and \ref{Camp}, we need to understand intersections between $\omega$-integral curves in a surface. 

\subsection{The discriminant locus} In this subsection $X$ will be a smooth projective surface, $\mathcal{L}$ an invertible sheaf, $r$ a positive integer, and $\omega\in H^0(X,\mathcal{L}\otimes S^r\Omega^1_X)$ a nonzero symmetric differential. 

Given a closed point $P$ in $X$ with inclusion map $i\colon\{P\}\to X$, we can choose a generator $\gamma$ of $i^*\mathcal{L}$ and a $\mathbb{C}$-basis $\{u_1,u_2\}$ of $i^*\Omega^1_X$. With these choices, the image of $\omega(P)$ via the canonical isomorphism $$i^*(\mathcal{L}\otimes_{\mathcal{O}_X}S^r\Omega^1_X)\to i^*\mathcal{L}\otimes S^r i^*\Omega^1_X$$ is of the form $\gamma \otimes H_{\omega,P,\gamma,u_1,u_2}(u_1,u_2)$, with $H_{\omega,P,\gamma,u_1,u_2}(\textbf{x,y})\in\mathbb{C}[\textbf{x,y}]$.

\begin{definition}\label{defdisc} 
The \emph{discriminant locus} of $\omega$ is
$$\Delta(\omega)=\{P\in X: H_{\omega,P,\gamma,u_1,u_2}(\textbf{x,y})\mbox{ has repeated factors or is the zero polynomial}\}.$$
\end{definition}

\begin{lemma}
The set $\Delta(\omega)$ is well defined, that is, it depends only on $\omega$.
\end{lemma}

\begin{proof}
Given a closed point $P\in X$, let $H:=H_{\omega,P,\gamma,u_1,u_2}(\textbf{x,y})$ be the polynomial associated to a choice of generator $\gamma$ of $i^*\mathcal{L}$ and a basis $\{u_1,u_2\}$ of $i^*\Omega^1_X$. Choosing a different basis $\{v_1,v_2\}$ of $i^*\Omega^1_X$ only modifies the polynomial $H$ by a linear change of variables, thus $H_{\omega,P,\gamma,v_1,v_2}$ has repeated factors if and only if $H$ has repeated factors. Let $\gamma'$ be another generator of $i^*\mathcal{L}$. We have that $H_{\omega,P,\gamma',v_1,v_2}$ is a multiple by a constant of $H_{\omega,P,\gamma,v_1,v_2}$. Thus $H_{\omega,P,\gamma',v_1,v_2}$ has repeated factors if and only if $H$ has repeated factors. Therefore $\Delta(\omega)$ is well defined.
\end{proof}

We will show that intersections between $\omega$-integral curves are of the same type, unless they intersect at the $\Delta(\omega)$. First we find conditions under which $\Delta(\omega)$ is a proper Zariski-closed set.

\begin{definition} 
An \emph{enhanced trivialization for $\mathcal{L}$} is a triple $\mathscr{G}_{\mathcal{L}}=(U,\alpha,\{u,v\})$, where
\begin{enumerate}
\item $U\subseteq X$ is a nonempty affine open set and $\alpha\colon\mathcal{O}_X|_U\xrightarrow{\sim}\mathcal{L}|_U$ is an isomorphism, 
\item $\{u,v\}\subseteq\mathrm{O}_X(U)$ satisfies that for each $Q\in U$, the pair $u-u(Q),v-v(Q)$ is a set of local parameters of $X$ at $Q$.
\end{enumerate}
\end{definition}

Note that since $U$ is affine, condition (2) is equivalent to the existence of $u,v\in\mathcal{O}_X(U)$ such that $du,dv$ are a basis of $\Omega^1_X|_U=\Omega^1_X(U)$ as an $\mathcal{O}_X|_U$-module. Hence from an enhanced trivialization $\mathscr{G}=(U,\alpha,\{u,v\})$, the image of $\omega|_U$ via the isomorphism $H^0(U,\mathcal{L}\otimes S^r\Omega^1_X)\cong H^0(U,S^r\Omega^1_X)$ induced by $\alpha$ has the form 
$$\sum_{j=0}^rA_j(du)^{r-j}(dv)^j,$$
with $A_j\in \mathcal{O}_U(U)$ for each $j$. Since $\{(du)^{r-j}(dv)^j\}_{j=0}^r$ form a basis of $S^r(\Omega^1_X|_U)$, we have uniqueness for the coefficients $A_j$. 
We denote by $K_{\omega,\mathscr{G}}$ the polynomial $\sum_{j=0}^rA_j\textbf{x}^{r-j}\textbf{y}^j\in \mathcal{O}_X(U)[\textbf{x,y}]$ and we denote $\delta_{\omega,\mathscr{G}}=\mathrm{disc}(K_{\omega,\mathscr{G}})\in \mathcal{O}_X(U)$. 

\begin{lemma}\label{delvu}
Let $\mathscr{G}=(U,\alpha,\{u,v\})$ be an enhanced trivialization for $\mathcal{L}$. With $\mathbb{V}_U(\delta_{\omega,\mathscr{G}})$ the zero locus of $\delta_{\omega,\mathscr{G}}$ in $U$, we have $\Delta(\omega)\cap U=\mathbb{V}_U(\delta_{\omega,\mathscr{G}})$. 
\end{lemma}

\begin{proof}
Let $P\in U$, and let $i\colon\{P\}\to U$ be the inclusion map. Let $s=\alpha^{-1}(1)\in\mathcal{L}(U)$, which is a generator of $\mathcal{L}|_{U}$, since $1$ generates $\mathcal{O}_X|_U$. Let $\gamma=i^*s$, this is a generator of $i^*\mathcal{L}$. Let $u_1=i^*du$, $u_2=i^*dv$, this is a $\mathbb{C}$-basis of $i^*\Omega^1_X$. Here $$H_{\omega,P,\gamma, u_1,u_2}=\sum_{j=0}^r i^*(A_j)\textbf{x}^{r-j}\textbf{y}^j=\sum_{j=0}^rA_j(P)\textbf{x}^{r-j}\textbf{y}^j\in\mathbb{C}[\textbf{x,y}]_{(r)},$$
with $A_i\in \mathcal{O}_X(U)$ as in the definition of $K_{\omega,\mathscr{G}}\in\mathcal{O}_X(U)[\textbf{x,y}]$.

Taking discriminants of polynomials of fixed degree respects ring morphisms on the coefficients, so $\mathrm{disc}(H_{\omega,P,\gamma,u_1,u_2})=i^*(\delta_{\omega,\mathscr{G}})=\delta_{\omega,\mathscr{G}}(P)\in\mathbb{C}$. Therefore, $\delta_{\omega,\mathscr{G}}(P)=0$ if and only if $H_{\omega,P,\gamma,u_1,u_2}$ has a repeated factor, that is, $P\in\Delta(\omega)$.
\end{proof}

\begin{definition}\label{reduced}
Let $\eta$ be the generic point of $X$. Fix an isomorphism $\epsilon\colon\mathcal{L}_\eta\xrightarrow{\sim}k(X)$ and $k(X)$-generators $df_1,df_2$ for $\Omega^1_{X,\eta}$. This induces an injective map $$H^0(X,\mathcal{L}\otimes S^r\Omega^1_X)\to\mathcal{L}_\eta\otimes S^r\Omega^1_{X,\eta}\cong S^r\Omega^1_{X,\eta}\cong k(X)[\textbf{x,y}]_{(r)}\subseteq k(X)[\textbf{x,y}].$$ We say that $\omega\in H^0(X,\mathcal{L}\otimes S^r\Omega^1_X)$ is \emph{reduced} if its image in $k(X)[\textbf{x,y}]$ is a reduced polynomial.
\end{definition}

The choice of generators $df_1,df_2$ corresponds to a $k(X)$-linear change of the variables $\textbf{x}$ and $\textbf{y}$. The choice of $\epsilon$ corresponds to multiplication by a scalar from $k(X)$. Hence the notion of $\omega$ being reduced is well-defined.

\begin{corollary}\label{66}
The set $\Delta(\omega)$ is Zariski-closed in $X$. Furthermore, $\omega$ is reduced if and only if $\Delta(\omega)$ is a proper Zariski-closed set.
\end{corollary}

\begin{proof}
We can cover $X$ with open sets $U$ coming from enhanced trivializations. From Lemma \ref{delvu}, we obtain that $\Delta(\omega)$ is Zariski-closed. 

Let $\mathscr{G}=(U,\alpha,\{u,v\})$ be any enhanced trivialization for $\mathcal{L}$. Since $U$ is a neighbourhood of $\eta$, we can use $\mathscr{G}$ to compute whether $\omega$ is reduced or not. We get that $\omega$ is reduced if and only if $\delta_{\omega,\mathscr{G}}\in\mathcal{O}_X(U)$ is not the zero function. 
\end{proof}


\subsection{Intersections between $\omega$-integral curves} 

Let $X$ be a smooth surface and let $P\in X$. Let $m_{X,P}$ be the maximal ideal of the local ring $\mathcal{O}_{X,P}$. Fixing a system of local parameters at $P$ we have an isomorphism $\hat{\mathcal{O}}_{X,P}\cong \mathbb{C}[[\bf{x,y}]]$, where $\hat{\mathcal{O}}_{X,P}$ is the $m_{X,P}$-adic completion of $\mathcal{O}_{X,P}$. 

For a curve $C\subseteq X$ passing through $P$, consider its associated principal prime ideal $(f)\in\mathcal{O}_{X,P}$. A \emph{branch of $C$ at $P$} is an ideal $\mathfrak{f}$ in $\hat{\mathcal{O}}_{X,P}$ satisfying $\mathfrak{f}\cap \mathcal{O}_{X,P}=(f)$.
With $\hat{m}_{X,P}$ the completion of $m_{X,P}$, we define the \emph{multiplicity of $C$ at $P$} $$\mathrm{mult}_P(C)=\sum_{\mathfrak{f}|(f)}\max\{i:\mathfrak{f}\subseteq \hat{m}^i_{X,P}\}=\max\{i:f\in m^i_{X,P}\}.$$

Since an element in $\hat{\mathcal{O}}_{X,P}$ can be a branch for at most one curve, it is useful to bound the number of branches at $P$ related to certain type of curves in order to bound the number of these curves passing through $P$. Let $\mathcal{L}$ be an invertible sheaf on $X$, let $r$ be a nonnegative integer, and let $\omega\in H^0(X,\mathcal{L}\otimes S^r\Omega^1_X)$ be a nonzero symmetric differential. By Corollary 3.76 in \cite{thesis} we have 

\begin{lemma}\label{lemth}
If $\omega$ is reduced, then for any given point $P\in X\setminus\Delta(\omega)$ there are at most $r$ $\omega$-integral curves passing through $P$. More precisely, the sum of the multiplicities $\mathrm{mult}_P(C)$ for all $\omega$-integral curves $C$ passing through $P$ is at most $r$, and they meet transversally at $P$.
\end{lemma}

This result is also implicit in the proof of Lemma 2.7 in \cite{V}, and we give a sketch of proof here: Given $P\in X\setminus\Delta(\omega)$, we choose an enhanced trivialization $(U,\alpha,\{u,v\})$ for $\mathcal{L}$ with $P\in U$. The image of $\omega|_U$ by $\alpha$ can be written as $\sum_{j=0}^rA_j(du)^{r-j}(dv)^j\in H^0(U,S^r\Omega^1_X)$, where $A_j\in \mathcal{O}_{X}(U)$.
It is mapped to $S^r\hat{\Omega}^1_{\mathcal{O}_{X,P}/\mathbb{C}}$ by localization and completion, and then (as in p.\ 138 of \cite{thesis}) to the universally finite differential algebra $S^r\tilde{\Omega}^1_{\hat{\mathcal{O}}_{X,P}/\mathbb{C}}\cong S^r\tilde{\Omega}^1_{\mathbb{C}[[\textbf{x},\textbf{y}]]/\mathbb{C}}\cong S^r(\mathbb{C}[[\textbf{x},\textbf{y}]]d\textbf{x}\oplus\mathbb{C}[[\textbf{x},\textbf{y}]]d\textbf{y})$ (see \cite{Kun86} for details), thus obtaining an expression of the form $$\tilde{\omega}_P=\sum_{j=0}^r \tilde{A}_j(d\textbf{x})^{r-j}(d\textbf{y})^j,$$ where $\tilde{A}_j\in \mathbb{C}[[\textbf{x},\textbf{y}]]$.

From Theorem 3.66 in \cite{thesis} we know that a curve $C\subseteq X$ passing through $P$ is $\omega$-integral if and only if its branches at $P$ are solutions of the differential equation $\tilde\omega_P=0$. 
Since $P$ is not in $\mathbb{V}_U(\delta_{\omega,\mathcal{G}})=\Delta(\omega)\cap U$ (cf.\ Lemma \ref{delvu}), we obtain that the linear factors of $$\tilde{\omega}_P(0,0)=\sum_{j=0}^r\tilde{A}_j(0,0)(d\textbf{x})^{r-j}(d\textbf{y})^j\in\mathbb{C}[d\textbf{x},d\textbf{y}]_{(r)}$$ are non-proportional over $\mathbb{C}$.

By applying Hensel's Lemma we obtain that $\tilde{\omega}_P$ factors into linear terms $B_jd\textbf{x}+C_jd\textbf{y}$ with $B_j,C_j\in \mathbb{C}[[\textbf{x},\textbf{y}]]$ such that for each $j$ at least one of $B_j(0,0)$ or $C_j(0,0)$ is not zero.
By Theorem 2 in \cite{S}, we obtain that each linear factor $B_jd\textbf{x}+C_jd\textbf{y}$ gives us exactly one branch solution $\mathfrak{f}_j$, and the solution branches $\mathfrak{f}_j$ have different tangents from the fact that linear factors of $\tilde{\omega}_P(0,0)$ are non-proportional. Thus for each point $P\in X$ outside of $\Delta(\omega)$, we know that there are at most $r$ $\omega$-integral curves passing through $P$, and they intersect transversally.


We now continue our study with the case of points in $\Delta(\omega)$, where the situation is quite different. 

\begin{definition}\label{defdeg}
Let $X$ be a smooth surface and $0\neq\omega\in H^0(X,\mathcal{L}\otimes S^r\Omega^1_X)$. A point $P\in X$ is \emph{degenerate for $\omega$} if there are infinitely many $\omega$-integral curves passing through $P$.
\end{definition}

Given an open set $U\subseteq X$, we denote the set of degenerate points of $\omega$ in $U$ by $\mathrm{DP}_U(\omega)$. By our discussion above, we know that $\mathrm{DP}_X(\omega)\subseteq \Delta(\omega)$.

\begin{lemma}\label{lemdp}
Let $0\neq \omega\in H^0(X,\mathcal{L}\otimes S^r\Omega^1_X)$. The set $\mathrm{DP}_X(\omega)$ is finite. Furthermore, there is a generically finite morphism of surfaces $\pi\colon Y\to X$ such that $\mathrm{DP}_Y(\pi^\bullet\omega)$ is empty.
\end{lemma}

\begin{proof}
Let $\eta$ be the generic point of $X$. Fix an isomorphism $\psi:\mathcal{L}_\eta\stackrel{\sim}{\to} k(X)$ and generators $df_1,df_2$ of $\Omega^1_{X,\eta}$ over $k(X)$. From these choices we obtain an injective map $$H^0(X,\mathcal{L}\otimes S^r\Omega^1_X)\to \mathcal{L}_\eta\otimes S^r\Omega^1_{X,\eta}\cong S^r\Omega^1_{X,\eta}\cong k(X)[\textbf{X,Y}]_{(r)}\subseteq k(X)[\textbf{X,Y}].$$
Denote by $H(\textbf{X,Y})\in k(X)[\textbf{X,Y}]$ the image of $\omega$ by this map. 
Let $F/k(X)$ be a finite extension such that the homogeneous degree $r$ polynomial $H(\textbf{X,Y})$ factors into linear terms over $F$. 
Let $X'$ be a smooth projective surface with generic point $\eta'$, and $\tau\colon X'\to X$ a generically finite morphism of surfaces with $k(X')=F$ and $\tau$ inducing the inclusion $k(X)\hookrightarrow F$. Then the image of $\tau^\bullet\omega\in H^0(X',\tau^*\mathcal{L}\otimes S^r\Omega^1_{X'})$ in $S^r\Omega^1_{X',\eta'}$ (via the map induced by $\psi$) factors as $\omega_1\cdots \omega_r$, with $\omega_j\in \Omega^1_{X',\eta'}$.

Given a point $Q\in X'$, we can choose an open neighborhood $U$ of $Q$ and rational functions $h_1,\ldots,h_r\in F^\times$ such that $h_1\omega_1,\ldots,h_r\omega_r$ define regular sections of $\Omega^1_{X'}$ on $U$, with the property that they have no zeroes or poles along curves in $U$, thus for each $j$ we have $h_j\omega_j\in H^0(U,\Omega^1_{X'})$ and its zero locus in $U$ will consist of at most finitely many points.

Each $h_j$ has finitely many curves that are zeroes or poles of it, so given any other point $Q'\in U$, the set of $\tau^\bullet\omega$-integral curves through $Q'$ is the same, up to finitely many curves, as the union of the sets of $h_j\omega_j$-integral curves through $Q'$, for $j=1,\ldots,r$. Therefore $\mathrm{DP}_U(\tau^\bullet\omega)=\cup_{j=1}^r\mathrm{DP}_U(h_j\omega_j)$. 

From Theorem 2 in \cite{S}, we know that each $\mathrm{DP}_U(h_j\omega_j)$ is contained in the zero locus of $h_j\omega_j$ on $U$, and that these degenerate points can be algorithmically resolved after a sequence of blow-ups (cf.\ \cite{S}).

Covering $X'$ with finitely many open sets of this type, this argument shows that $DP_{X'}(\omega)$ consists of finitely many points.
Finally, by a local analysis as in Theorem 3.35 in \cite{thesis}, we have that $$\mathrm{DP}_X(\omega)\subseteq \tau (DP_{X'}(\tau^\bullet\omega))\cup\tau(\mbox{contracted fibres of }\tau:X'\to X),$$ thus $DP_X(\omega)$ is finite. The map $\pi\colon Y\to X$ is obtained upon a sequence of blow-ups on $X'$.
\end{proof}

\begin{proof}[Proof of Theorem \ref{vojta}] 
From Lemmas \ref{lemth} and \ref{lemdp}, we know that the set $Z_{\omega,D}$ is the union of finitely many $\omega$-integral curves, hence $Z_{\omega,D}$ is a proper Zariski-closed set.
Let $\phi\colon C\to X$ be a nonconstant morphism with $\phi(C)\not\subseteq Z_{\omega,D}$ for which Inequality \eqref{1k} fails. Let $a,b\in\mathbb{Z}_{>0}$ be such that $$q\epsilon\geq \frac{a}{b}\geq\max_{1\leq i\leq q}\sigma(\mathcal{O}(D_i),\mathcal{L}).$$
Then $\mathcal{O}(D)^{\otimes a}\otimes(\mathcal{L}^\vee)^{\otimes bq}\cong \bigotimes_{i=1}^q (\mathcal{O}(D_i)^{\otimes a}\otimes (\mathcal{L}^\vee)^{\otimes b})$ is ample.
From this we have $$0 <\frac{1}{bq}(ah_{\mathcal{O}(D)}(\phi)-bqh_\mathcal{L}(\phi))\leq \epsilon h_{\mathcal{O}(D)}(\phi)-h_\mathcal{L}(\phi).$$ Since Inequality \eqref{1k} fails for $\phi$, we obtain
$$h_{X,\mathcal{O}(D)}(\phi)-h_{X,\mathcal{L}}(\phi)\geq (1-\epsilon)h_{X,\mathcal{O}(D)}(\phi)>\sum_{j=1}^qN^{(1)}(D_j,\phi)+2r\max\{0,g(\phi(C))-1\},$$
hence $\phi(C)$ is an $\omega$-integral curve by Theorem \ref{main}.

Recall that $\phi(C)\not\subseteq Z_{\omega,D}$. By Lemma \ref{lemth}, we have that for each $i$ the intersections between $\phi(C)$ and $D_i$ can only be transversal and, moreover, for each point there is at most one $i$ for which $i_{\phi(Q)}(\phi(C),D_j)$ is different from $0$. By this and Proposition \ref{eio}, we obtain 
\begin{eqnarray*}
h_{X,\mathcal{O}(D)}(\phi)-\sum_{j=1}^qN^{(1)}(D_j,\phi) &=& \sum_{Q\in C}\sum_{j=1}^q \mathrm{ord}_Q(\phi^*E_{D_j,\phi(Q)})-\sum_{Q\in C}\sum_{j=1}^q \min\{1,\mathrm{ord}_Q(\phi^*E_{D_j,\phi(Q)})\}\cr
&=&\sum_{Q\in C}\sum_{j=1}^q \max\{\mathrm{ord}_Q(\phi^*E_{D_j,\phi(Q)})-1,0\}\cr
&=& \sum_{Q\in C}\sum_{j=1}^q \max\{i_{\phi(Q)}(\phi(C),D_j)\cdot e_{Q/\phi(Q)}(\phi)-1,0\}\cr
&\leq& \sum_{Q\in C} e_{Q/\phi(Q)}-1.
\end{eqnarray*}
Since Inequality \eqref{1k} fails for $\phi$, we can apply the Riemann-Hurwitz formula to $\tilde{\phi}:C\to \widetilde{\phi(C)}$ to get
$$\epsilon h_{X,\mathcal{O}(D)}(\phi)+2r\max\{0,g(C)-1\}\leq \sum_{Q\in C}(e_{Q/\tilde{\phi}(Q)}(\tilde{\phi})-1)=2g(C)-2-(2g(\widetilde{\phi(C)})-2)\deg(\tilde{\phi}).$$
As each $D_i$ is ample, we have 
$$h_{\mathcal{O}(D)}(\phi)=\deg(\tilde{\phi})h_{\mathcal{O}(D)}(\nu)=\deg(\tilde{\phi})\sum_{j=1}^q(\phi(C).D_j)_X\geq q\deg(\tilde{\phi}),$$ 
hence
$$(\epsilon q+2g(\tilde{\phi(C)})-2)\deg(\tilde{\phi})\leq 2(g(C)-1)-2r\max\{0,g(C)-1\}.$$
Since $\epsilon q>2$ and $g(\widetilde{\phi(C)})\geq 0$, we get $g(C)-1> r\max\{0,g(C)-1\}\geq g(C)-1$ which is not possible. Therefore Inequality \eqref{1k} must hold for $\phi$.
\end{proof}

\begin{proof}[Proof of Theorem \ref{Camp}] Since $-D_\mathcal{L}+(D-D_{m\overline{\epsilon}})$ is big, there is $M\geq1$, an ample line sheaf $\mathcal{A}$ and an effective divisor $E\geq 0$ such that
$$\mathcal{O}(M(-D_\mathcal{L}+D-D_{m\overline{\epsilon}})) \simeq \mathcal{A}\otimes \mathcal{O}(E).$$
For any curve $C\subseteq X$ with image not contained in $Z:=\Delta(\omega)\cup\mathrm{supp}(E)\cup\mathrm{supp}(D)$ we have
$$h_{\mathcal{A}}(\nu_C)\leq M\cdot \left(-h_{\mathcal{L}}(\nu_C) + h_{\mathcal{O}(D)}(\nu_C) - \sum_{j=1}^q m\epsilon_j h_{\mathcal{O}(D_j)}(\nu_C) \right)$$ because $E$ is effective.
Let $C\subseteq X$ be an $(X,D_{\overline{\epsilon}})$-Campana curve with image not contained in $Z$. The $(X,D_{\overline{\epsilon}})$-Campana condition gives us $m\epsilon_j h_{\mathcal{O}(D_j)}(\nu_C)\geq mN^{(1)}(D_j,\nu_C)\geq N^{(m)}(D_j,\nu_C) $ for each $j$. If $C$ is not $\omega$-integral, from Theorem \ref{main} we have
$$h_{\mathcal{O}(D)}(\nu_C)-h_{\mathcal{L}}(\nu_C) \leq \sum_{j=1}^q N^{(m)}(D_j, \nu_C) + 2r\max\{0,g(C)-1\}.$$
This leads to a contradiction with the previous estimates when $B=2Mr$. Therefore $C$ must be $\omega$-integral.
Since $C$ is an $(X,D_{\overline{\epsilon}})$-Campana curve, by Lemma \ref{lemth} it can only intersect $D$ at points in $\Delta(\omega)$, because $\frac{1}{m}>\epsilon_j$ for each $j$. Since $C$ intersects $D$, we obtain that $C\in Z_{\bar{\epsilon},\omega,D}$.

The set $Z_{\bar{\epsilon},D,\omega}$ is a proper Zariski closed subset of $X$ by Corollary \ref{66} and from the fact that $D\cap\mathrm{DP}_X(\omega)=\emptyset$.
\end{proof}


\section{Vojta's conjecture for quadratic families of lines in $\mathbb{P}^2$}\label{7}

The aim of this section is to work out in detail the case where the target divisor is formed by a quadratic family of lines in $\mathbb{P}^2$, giving a concrete description of the exceptional set. 

\begin{definition}
A \emph{quadratic family of lines in }$\mathbb{P}^2$ is a smooth conic in $(\mathbb{P}^2)^\vee$. Equivalently, by parametrization of that conic, it is a collection of lines $H_{[s:t]}: f(s,t)x+g(s,t)y+h(s,t)z=0$ in $\mathbb{P}^2$, for $[s:t]\in\mathbb{P}^1$, where $f,g,h$ are homogeneous of degree $2$, and not all the $H_{[s:t]}$ pass through the same point (this means that the corresponding degree $2$ curves in $(\mathbb{P}^2)^\vee$ is smooth).
\end{definition}

\begin{proof}[Proof of Corollary \ref{quad2}] Let $C_{[s:t]}:f(s,t)x+g(s,t)y+h(s,t)z=0$ be a quadratic family of lines, with $L_i=C_{[s_i:t_i]}$. We know $\max\{\deg(f),\deg(g),\deg(h)\}=2$, so we can write $$C_{[s:t]}:s^2\cdot a(x,y,z)+st\cdot b(x,y,z)+t^2\cdot c(x,y,z)=0,$$ with $a,b,c$ linear on $x,y,z$.
Applying an automorphism of $\mathbb{P}^2$ does not change the result, so we can assume that $C_{[s:t]}:s^2x+st y+t^2z=0$.
Here we are applying the automorphism $\alpha\colon \mathbb{P}^2\to\mathbb{P}^2$ given by $x\mapsto a(x,y,z)$, $y\mapsto b(x,y,z)$, $z\mapsto c(x,y,z)$.

Let $\omega\in H^0(\mathbb{P}^2,\mathcal{O}(4)\otimes S^2\Omega^1_{\mathbb{P}^2})$ be a symmetric differential which in $U:=\{x\neq 0\}$ looks like $$dxdx-ydxdy+xdydy.$$
It can be verified that the lines $C_{[s:t]}$ are $\omega$-integral.
We have $\Delta(\omega)\cap U=\mathbb{V}_U(y^2-4z)$, and with $V:=\{z\neq 0\}$ we have $\Delta(\omega)\cap V=\mathbb{V}_V(y^2-4x)$. Therefore $\Delta(\omega)=\{y^2-4xz\}$, which is an $\omega$-integral curve.
By Lemma \ref{lemth}, we know that at each point $P\notin \Delta(\omega)$ at most two $\omega$-integral curves intersect. Therefore the $\omega$-integral curves are $\Delta(\omega)$ and the lines $C_{[s:t]}$.

We will now compute the special set $Z_{\omega,D}$. The lines $L_i$ are trivially contained in $Z_{\omega,D}$, since every pair of lines in the quadratic family intersect. Similarly, the curve $\Delta(\omega)$ is also in $Z_{\omega,D}$, as it is the envelope of the quadratic family.
A line $C_{[s:t]}$ different from the components of $D$ is not in $Z_{\omega,D}$, by Lemma \ref{lemth} and because two lines in the quadratic family intersect outside $\Delta(\omega)$.
Therefore $Z_{\omega,D}=Z$.

Note that $\sigma(\mathcal{O}(L_i),\mathcal{O}(4))=\sigma(\mathcal{O}(1),\mathcal{O}(4))=4$, thus the condition in Theorem \ref{vojta} is $q\epsilon>4$. Applying Theorem \ref{vojta} we obtain the result.
\end{proof}



\begin{thebibliography}{9}

\bibitem[Abr09]{Abr09} D. Abramovich, \emph{Birational geometry for number theorists.} in Arithmetic geometry, Clay Mathematics Proceedings, vol. 8, American Mathematical Society, 2009, p. 335--373.

\bibitem[AV18]{AV18} D. Abramovich, A. V\'arilly-Alvarado, \emph{Campana points, Vojta's conjecture, and level structures on semistable abelian varieties.} J. Th\'eor. Nombres Bordeaux 30 (2018), no. 2, 525--532. 

\bibitem[BD19]{BD19} D. Brotbek, Y. Deng, \emph{Kobayashi hyperbolicity of the complements of general hypersurfaces of high degree.} Geom. Funct. Anal. 29 (2019), no. 3, 690--750.

\bibitem[BTV19]{BTV19} N. Bruin, J. Thomas, A. V\'arilly-Alvarado, \emph{Explicit computation of symmetric differentials and its application to quasi-hyperbolicity.} To appear in Algebra \& Number Theory. arXiv:1912.08908 

\bibitem[Cam05]{Cam05} F. Campana, \emph{Fibres multiples sur les surfaces: aspects geom\'etriques, hyperboliques et arithm\'etiques.} Manuscr. Math. 117 (2005), no. 4, p. 429--461.

\bibitem[CP21]{CP21} J. Caro, H. Pasten, \emph{A Chabauty-Coleman bound for surfaces.} arXiv:2102.01055

\bibitem[CZ04]{CZ04} P. Corvaja, U. Zannier, \emph{On integral points on surfaces.} Ann. of Math. (2), 160(2):705--726, 2004.

\bibitem[Dem97]{Dem97} J.-P. Demailly, \emph{Algebraic criteria for Kobayashi hyperbolic projective varieties and jet differentials.} Algebraic geometry--Santa Cruz 1995, 285--360, Proc. Sympos. Pure Math., 62, Part 2, Amer. Math. Soc., Providence, RI, 1997. 

\bibitem[Gar15]{thesis} N. Garcia-Fritz, \emph{Curves of low genus on surfaces and applications to Diophantine problems}. Thesis (Ph.D.)--Queen's University (Canada). 2015. \texttt{http://hdl.handle.net/1974/13545}

\bibitem[Gar18a]{Gar18} N. Garcia-Fritz, \emph{Sequences of powers with second differences equal to two and hyperbolicity.} Trans. Amer. Math. Soc. 370 (2018), no. 5, 3441--3466.

\bibitem[Gas10]{Gas10} C. Gasbarri, \emph{The strong abc conjecture over function fields (after McQuillan and Yamanoi).} S\'eminaire Bourbaki. Vol. 2007/2008. Ast\'erisque No. 326 (2009), Exp. No. 989, viii, 219-256 (2010).

\bibitem[GG80]{GG} M. Green, P. Griffiths, \emph{Two applications of algebraic geometry to entire holomorphic mappings}. The Chern Symposium 1979 (Proc. Internat. Sympos., Berkeley, Calif., 1979), pp. 41--74, Springer, New York-Berlin, 1980. 

\bibitem[EGA]{EGA} A. Grothendieck, J. Dieudonn\'e, \emph{\'El\'ements de g\'eom\'etrie alg\'ebrique}, I Grundlehren der Mathematik 166 (1971) Springer, II-IV Publ. Math. IHES 8 (1961) 11 (1961) 17 (1963) 20 (1964) 24 (1965) 28 (1966) 32 (1967). 

\bibitem[Har77]{H} R. Hartshorne, \emph{Algebraic geometry}. Graduate Texts in Mathematics, No. 52. Springer-Verlag, New York-Heidelberg, 1977.

\bibitem[Kun86]{Kun86} E. Kunz, \emph{K\"ahler differentials.} Advanced Lectures in Mathematics. Friedr. Vieweg \& Sohn, Braunschweig, 1986.

\bibitem[Nog97]{Nog97} J. Noguchi, \emph{Nevanlinna-Cartan theory over function fields and a Diophantine equation.} J. Reine Angew. Math. 487 (1997), 61--83.

\bibitem[RTW21]{RTW21} E. Rousseau, A. Turchet, J. T.-Y. Wang, \emph{Nonspecial varieties and generalised Lang-Vojta conjectures.} Forum Math. Sigma 9 (2021), Paper No. e11, 29 pp. 

\bibitem[Sei68]{S} Seidenberg, A. \emph{Reduction of singularities of the differential equation $Ady=Bdx$}. Amer. J. Math. 90 1968 248--269.

\bibitem[Voj87]{Voj87} P. Vojta, \emph{Diophantine approximations and value distribution theory.} Lecture Notes in Mathematics, 1239. Springer-Verlag, Berlin, 1987.

\bibitem[Voj91]{Voj91} P. Vojta, \emph{Diophantine approximation and Nevanlinna theory}. In J.-L. Colliot-Th\'el\`ene, K. Kato, P. Vojta, \emph{Arithmetic algebraic geometry}. Lectures from the Second C.I.M.E. Session held in Trento, June 24-July 2, 1991. Edited by E. Ballico. Lecture Notes in Mathematics, 1553. Springer-Verlag, Berlin, 1993.

\bibitem[Voj98]{Voj98} P. Vojta, \emph{A more general abc conjecture.} Internat. Math. Res. Notices 1998, no. 21, 1103-1116.

\bibitem[Voj00]{V} P. Vojta, \emph{Diagonal quadratic forms and Hilbert's tenth problem}. Hilbert's tenth problem: relations with arithmetic and algebraic geometry (Ghent, 1999), 261--274, Contemp. Math., 270, Amer. Math. Soc., Providence, RI, 2000.

\bibitem[Voj07]{V2} P. Vojta, \emph{Jets via Hasse-Schmidt derivations.} Diophantine geometry, 335--361, CRM Series, 4, Ed. Norm., Pisa, 2007.

\bibitem[Wan96]{Wan96} J. T.-Y. Wang, \emph{The truncated second main theorem of function fields.} J. Number Theory 58 (1996), no. 1, 139--157.

\bibitem[Wan04]{Wan04} J. T.-Y. Wang, \emph{An effective Schmidt subspace theorem over function fields,} Math. Z. 204 (2004), 811--844.

\bibitem[Yam04]{Yam04} K. Yamanoi, \emph{The second main theorem for small functions and related problems.} Acta Math. 192 (2004), 225--294.

\end{thebibliography}
\end{document}